\documentclass[reqno]{amsart}
\usepackage[utf8]{inputenc}
\usepackage{amssymb, amsmath, amsthm, color, enumerate, mathabx, mathrsfs}
\usepackage{bbm}
\usepackage{pdfpages}
\usepackage{tcolorbox}
\usepackage{esint} % for \fint
\usepackage{marginnote}
\usepackage{thmtools}
\usepackage{mathtools}

\usepackage{hyperref}

\numberwithin{equation}{section}
\theoremstyle{plain}
\newtheorem{theorem}{Theorem}[section]
\newtheorem{lemma}[theorem]{Lemma}
\newtheorem{proposition}[theorem]{Proposition}

\newtheorem{conjecture}[theorem]{Conjecture}

\theoremstyle{definition}
\newtheorem{definition}[theorem]{Definition}

\newtheorem{remark}[theorem]{Remark}

\RenewCommandCopy{\theHtheorem}{\thetheorem}
\RenewCommandCopy{\theHequation}{\theequation}
\RenewCommandCopy{\theHlemma}{\thelemma}
\RenewCommandCopy{\theHdefinition}{\thedefinition}
\RenewCommandCopy{\theHproposition}{\theproposition}

\newcommand{\N}{\mathbb{N}}

\newcommand{\bbP}{\mathbb{P}}

\newcommand{\R}{\mathbb{R}}

\newcommand{\T}{\mathbb{T}}

\newcommand{\Z}{\mathbb{Z}}

\newcommand{\cC}{\mathcal{C}}

\newcommand{\cE}{\mathcal{E}}

\newcommand{\cI}{\mathcal{I}}
\newcommand{\cJ}{\mathcal{J}}

\newcommand{\cT}{\mathcal{T}}

\newcommand{\fkA}{\mathfrak{A}}

\newcommand{\fkG}{\mathfrak{G}}

\newcommand{\fkM}{\mathfrak{M}}

\newcommand{\bdA}{\mathbf{A}}

\newcommand{\bdC}{\mathbf{C}}
\newcommand{\bdD}{\mathbf{D}}

\newcommand{\bdT}{\mathbf{T}}

\newcommand{\ud}{\mathrm{d}}
\def\inn#1#2{\langle#1,#2\rangle}

\newcommand{\supp}[1]{\mathrm{supp}_{#1}\,}

\newcommand{\dist}{\mathrm{dist}}
\newcommand{\high}{\mathrm{high}}
\newcommand{\low}{\mathrm{low}}

\newcommand{\wt}[1]{\widetilde{#1}}

\usepackage[normalem]{ulem}

\begin{document}

\title[Space curve $L^2$ Fourier restriction]{Remarks on space curve $L^2$ Fourier restriction}

\date{}

\begin{abstract} We re-examine and consider variants of the Fourier restriction theorem for space curves. In particular, we introduce a Stein--Tomas-type argument which recovers the sharp $L^2$ restriction theorem from classical oscillatory integral estimates, in the spirit of Arkhipov--Chubarikov--Karatsuba. In addition, we consider certain square function estimates which can be thought of as partial progress towards a space curve variant of the orthonormal Fourier restriction theorem of Frank--Sabin. These estimates are obtained by reducing the problem to a Kakeya-type maximal inequality, which we analyse using a formulation of the Bourgain--Guth broad/narrow method.
\end{abstract}

\author[J-C. Cuenin]{ Jean-Claude Cuenin }
\address{Jean-Claude Cuenin: Department of Mathematical Sciences, Loughborough University, Loughborough, Leicestershire, LE11 3TU, United Kingdom}
\email{J.Cuenin@lboro.ac.uk}

\author[J. Denson]{Jacob Denson}
\address{Jacob Denson: School of Mathematics and Maxwell Institute for Mathematical Sciences, James Clerk Maxwell Building, The King’s Buildings, Peter Guthrie Tait
Road, Edinburgh, EH9 3FD, UK}
\email{Jacob.Denson@ed.ac.uk}

\author[J. Hickman]{ Jonathan Hickman }
\address{Jonathan Hickman: School of Mathematics and Maxwell Institute for Mathematical Sciences, James Clerk Maxwell Building, The King's Buildings, Peter Guthrie Tait Road, Edinburgh, EH9 3FD, UK.}
\email{jonathan.hickman@ed.ac.uk}

\maketitle

%%%%%%%%%%%%%%%%%%%%%%%%%%%%%%%%%%%%%%%%%%%%%%%%%%%%%%%%%%%%%%%%%%%%%%%%%%%%%%%%%%%%%%%%%%%%%%%%

%    Introduction

%%%%%%%%%%%%%%%%%%%%%%%%%%%%%%%%%%%%%%%%%%%%%%%%%%%%%%%%%%%%%%%%%%%%%%%%%%%%%%%%%%%%%%%%%%%%%%%%

\section{Introduction}%\label{sec: intro}

%%%%%%%%%%%%%%%%%%%%%%%%%%%%%%%%%%%%%%%%%%%%%%%%%%%%%%%%%%%%%%%%%%%%%%%%%%%%%%%%%%%%%%%%%%%%%%%%

%   Fourier restriction for space curves

%%%%%%%%%%%%%%%%%%%%%%%%%%%%%%%%%%%%%%%%%%%%%%%%%%%%%%%%%%%%%%%%%%%%%%%%%%%%%%%%%%%%%%%%%%%%%%%%

\subsection{Fourier restriction for space curves} \label{sub:fourier_restriction_for_space_curves} We say a $C^{\infty}$ function $\gamma \colon [-2,2] \to \R^n$ is (a parametrisation of) a \textit{smooth, non-degenerate curve} if
\begin{equation}\label{eq: nondegenerate}
    |\gamma^{(1)}(s) \wedge \cdots \wedge \gamma^{(n)}(s)| \neq 0\quad \textrm{for all $s \in [-2,2]$.}
\end{equation}
Fix $\alpha \in C_c^\infty(\R)$ non-negative and not identically zero with $\supp{} \alpha \subseteq [-1,1]$, and define the \textit{Fourier extension operator}
\begin{equation} \label{eq:2eio1j2oiej12oilabel}
    \cE_{\gamma} g(x) \coloneq \int_{\R} e^{2 \pi i \gamma(t) \cdot x} g(t) \alpha(t)\,\ud t \qquad \textrm{for all $g \in L^1(\R) + L^{\infty}(\R)$;} 
\end{equation} 
The classical Fourier restriction theorem for space curves, due to Drury~\cite{Drury1985}, determines the $L^p$-$L^q$ mapping properties of $\cE_{\gamma}$. In particular, $\cE_\gamma$ is bounded from $L^p(\R)$ to $L^q(\R^n)$ if and only if
\begin{equation}\label{eq: Drury exp}
   \frac{n^2 + n + 2}{2} < q \leq \infty \qquad \textrm{and} \qquad q = \frac{n(n+1)}{2} p.
\end{equation}
Drury's proof \cite{Drury1985} combines a multilinear argument from \cite{Prestini1978, Christ1985} with an iteration procedure. 

In this note, we revisit the Fourier restriction theorem for space curves, focusing on $L^2$-based results: that is, we consider the $L^2$-$L^q$ mapping properties of $\cE_{\gamma}$. We are interested in two problems, discussed further in the following subsections. Namely, we obtain a new proof of the optimal $L^2$-$L^q$ estimate for $\cE_{\gamma}$, and prove new square function estimates for the Fourier extension of orthonormal wave packets.

%%%%%%%%%%%%%%%%%%%%%%%%%%%%%%%%%%%%%%%%%%%%%%%%%%%%%%%%%%%%%%%%%%%%%%%%%%%%%%%%%%%%%%%%%%%%%%%%

%   $L^2$ Fourier restriction via Arkhipov--Chubarikov--Karatsuba

%%%%%%%%%%%%%%%%%%%%%%%%%%%%%%%%%%%%%%%%%%%%%%%%%%%%%%%%%%%%%%%%%%%%%%%%%%%%%%%%%%%%%%%%%%%%%%%%

\subsection{\texorpdfstring{$L^2$ Fourier restriction via Arkhipov--Chubarikov--Karatsuba Estimates}{L2 Fourier restriction via Arkhipov-Chubarikov-Karatsuba Estimates}}\label{subsec: intro rest via ACK}

Given an interval $J \subseteq [-1,1]$, define
\begin{equation}\label{eq: H functional}
    H_{\gamma; J}(x) \coloneq \inf_{t \in J} \max_{1 \leq r \leq n} | \gamma^{(r)}(t) \cdot x |^{1/r}
\end{equation}
and suppose $b \in C_c^1(\R)$ satisfies $\supp{} b \subseteq J$. 
By a variant of the Arkhipov--Chubarikov--Karatsuba estimate \cite{ACK_book, ACK1979} (see Theorem~\ref{thm: ACK} below), the inequality\footnote{Given quantities $A,B \geq 0$, and a family of parameters $\theta$, we write $A \lesssim_\theta B$ or $B \gtrsim_{\theta} A$ to indicate that $A \leq C_{\theta} B$ for a constant $C_{\theta} > 0$ depending only on the dimension $n$ and the parameters $\theta$. We write $A \sim_{\theta} B$ to indicate $A \lesssim_{\theta} B$ and $B \lesssim_{\theta} A$. We let $O_{\theta}(B)$ denote a constant $A$ such that $A \lesssim_\theta B$, though different occurences of $O_\theta(B)$ will denote different constants.}
\begin{equation}\label{eq: ACK intro}
  \bigg|\int_J e^{2 \pi i x \cdot \gamma(t)} b(t)\,\ud t\bigg| \lesssim \|b'\|_{L^1(\R)} H_{\gamma, J}(x)^{-1}. 
\end{equation}
holds for all $x \in \R^n$. The oscillatory integral estimate \eqref{eq: ACK intro} is very sharp, in the sense that the reverse inequality also often holds (see \cite[pp.16-17]{ACK_book} for the precise details). This observation can be used to show the necessity of the first condition in \eqref{eq: Drury exp}. Indeed, it is possible to use \eqref{eq: ACK intro} (and, moreover, the accompanying reverse estimate) to show, for the indicator function $\chi_{[-1,1]}$, that
\begin{equation}\label{eq: ACK nec}
    \|\cE_{\gamma} \chi_{[-1,1]} \|_{L^q(\R^n)} < \infty \qquad \iff \qquad q > \frac{n^2 + n + 2}{2};
\end{equation}
see \cite[p.21, Theorem 1.3]{ACK_book} or the recent work of dall'Ara--Wright \cite{dAW} (which also provides an interesting new perspective on the Arkhipov--Chubarikov--Karatsuba theory) for details.

It is natural (see Remark~\ref{rmk: hypersurface} below) to ask whether the Arkhipov--Chubarikov--Karatsuba estimate \eqref{eq: ACK intro}, or its proof, can also be used to prove positive results for $\cE_{\gamma}$. Here we consider the $L^2$ restriction theory for $\cE_{\gamma}$ and develop a Stein--Tomas \cite[Chapter IX, \S2.1]{Stein_book} argument in this setting. In particular, define the exponent 
\begin{equation}\label{eq: crit q}
    q_n \coloneq n(n+1).
\end{equation}
By adapting the argument used to prove \eqref{eq: ACK intro}, in \S\ref{sec: rest via ACK} we recover the following $L^2$-based Fourier restriction inequality. 
\begin{theorem}\label{thm: L2 endpoint} Let $n \geq 2$ and $\gamma \colon [-2,2] \to \R^n$ be a smooth, non-degenerate curve. For $q \geq q_n$, we have
\begin{equation*}
  \|\cE_{\gamma} g\|_{L^{q_n}(\R^n)} \lesssim_{\gamma} \|g\|_{L^2(\R)}  \qquad \textrm{for all $g \in L^2(\R)$.}
\end{equation*}
\end{theorem}

Comparing with \eqref{eq: Drury exp}, we see this is sharp in the sense that the $q_n$ exponent cannot be replaced with any $q < q_n$. We emphasise that Theorem~\ref{thm: L2 endpoint} is not new, and indeed has been known since Prestini's work \cite{Prestini1978} from the 1970s. The main interest here is in the method of proof. 

The classical Stein--Tomas argument (see, for instance, \cite[Chapter 8, \S5]{SS_book} or \cite[Chapter IX, \S2.1]{Stein_book}) is a powerful tool for proving $L^2$-based Fourier restriction estimates, exploiting decay of the (inverse) Fourier transform of the underlying measure supported on the surface. In the setting of non-degenerate curves, a naive application of the Stein--Tomas argument gives very poor results, as the Fourier transform of the measure underlying $\cE_\gamma$ does not decay uniformly in all directions.

% $\widecheck{\sigma}_\gamma$, where $\sigma_\gamma$ is the measure supported on $\gamma$ such that $\cE_\gamma (\cE_\gamma)^*$ is given by convolution against $\widecheck{\sigma}_\gamma$. 

Recently, Carnovale, Fraser and de Orellana \cite{CFdO} studied, \textit{inter alia}, the Stein--Tomas argument for space curves. They considered a quantity called the \textit{Fourier spectrum}, which measures average decay, in a certain sense. In the setting of nondegenerate curves, the Fourier spectrum can be computed by appeal to the Arkhipov--Chubarikov--Karatsuba estimate \eqref{eq: ACK intro}. Using this, the authors of \cite{CFdO} were able to run a variant of the Stein--Tomas argument to establish $L^2$ bounds for $\cE_{\gamma}$. Whilst the argument of \cite{CFdO} does improve over the naive application of the Stein--Tomas method, the bounds obtained in \cite{CFdO} via the Fourier spectrum are still very far from optimal. By contrast, our argument proving Theorem~\ref{thm: L2 endpoint} achieves an optimal $L^2$ result via a Stein--Tomas-type approach. 

An alternative proof of Theorem~\ref{thm: L2 endpoint} in the spirit of the classical Stein--Tomas argument, not including the endpoint $L^2$-$L^{q_n}$ bound, can be derived from the methods of the sister paper \cite{CDH}. The approach used here, however, is much more direct than that in \cite{CDH}. It also has the advantage of explicitly relating the $L^2$ theory of Fourier extension to curves to the Arkhipov--Chubarikov--Karatsuba estimate \eqref{eq: ACK intro} used to show the necessary condition \eqref{eq: ACK nec}.

\begin{remark}\label{rmk: hypersurface} In order to better understand the underlying philosophy here, it is useful to compare the above discussion with the classical $L^2$ restriction theory for hypersurfaces. Let $\Sigma \colon [-2,2]^{n-1} \to \R^n$ be (a parametrisation of) a smooth hypersurface of non-vanishing Gaussian curvature. Fix $\alpha \in C_c^\infty(\R^{n-1})$ non-negative and not identically zero with $\supp \alpha \subseteq [-1,1]$, and define the extension operator
\[ \cE_\Sigma g(\xi) = \int_{\R^n} e^{2 \pi i\, \Sigma(y) \cdot x} g(y) \alpha(y)\; \ud y \quad\text{for all $g \in L^1(\R^n) + L^\infty(\R^n)$} \]
analogous to the extension operators introduced in \S\ref{sub:fourier_restriction_for_space_curves}. Then consider the pushforward measure $\sigma_\Sigma = \Sigma_*( \alpha\; \ud y )$.

The classical van der Corput oscillatory integral estimate then tells us
\begin{equation}\label{eq: van der Corput}
    |\widecheck{\sigma}_{\Sigma}(x)| \lesssim (1 + |x|)^{- \frac{n-1}{2}}. 
\end{equation}
This estimate is sharp, in the sense that the reverse inequality also holds on an open cone in $\R^n$. Using this observation, it is possible to show
\begin{equation*}
    \|\cE_{\Sigma} \chi_{[-1,1]^{n-1}} \|_{L^q(\R^n)} < \infty \qquad \iff \qquad q > 2 \cdot \frac{n}{n-1},
\end{equation*}
which provides a necessary condition on $q$ for the Fourier restriction conjecture (see, for instance, \cite{Tao2004}). On the other hand, \eqref{eq: van der Corput} can also be used in the classical Stein--Tomas argument to prove the sharp $L^2$ Fourier restriction theorem for $\cE_{\Sigma}$. 

Our proof of Theorem~\ref{thm: L2 endpoint} mirrors the situation for hypersurfaces. It shows that the oscillatory integral estimates which are used to show the necessary condition \eqref{eq: ACK nec} can also be used to prove the sharp $L^2$ extension estimates for $\cE_\gamma$.
\end{remark}

%%%%%%%%%%%%%%%%%%%%%%%%%%%%%%%%%%%%%%%%%%%%%%%%%%%%%%%%%%%%%%%%%%%%%%%%%%%%%%%%%%%%%%%%%%%%%%%%

%   Square function bounds related to orthonormal Fourier restriction}

%%%%%%%%%%%%%%%%%%%%%%%%%%%%%%%%%%%%%%%%%%%%%%%%%%%%%%%%%%%%%%%%%%%%%%%%%%%%%%%%%%%%%%%%%%%%%%%%

\subsection{Square function bounds related to orthonormal Fourier restriction} Recently, the authors \cite{CDH} investigated variants of the $L^2$ Fourier restriction estimate for space curves involving orthonormal systems. This work complements Frank--Sabin's work \cite{FS2017} on orthonormal Fourier restriction for hypersurfaces. We remark that the original multilinear approach \cite{Prestini1978, Christ1985, Drury1985} to the Fourier restriction for curves appears to be incompatible with the theory for orthonormal systems. In particular, \cite{CDH} required the development of Stein--Tomas-type arguments, exploiting Fourier decay and Hilbert space structure. Indeed, this was a major motivation for the line of investigation discussed in \S\ref{subsec: intro rest via ACK} above.

By testing against simple examples (see \cite[Proposition 4.1]{CDH}), for $q_n$ as in \eqref{eq: crit q}, it is natural to conjecture the following. 

\begin{conjecture}[{\cite[Conjecture 1.2]{CDH}}]\label{conj: CDH} 
Let $n \geq 2$ and $\gamma \colon [-2,2] \to \R^n$ be a smooth, non-degenerate curve. For $q \geq q_n$ and $1 \leq r \leq q/2n$, the inequality
\begin{equation*}%\label{eq: ON Fourier rest 2}
    \left\|\sum_{j \in J}\nu_j |\cE_{\gamma} g_j|^2 \right\|_{L^{q/2}(\R^n)} \lesssim_{\gamma, \alpha} \left( \sum_{j \in J} |\nu_j|^r \right)^{1/r}
\end{equation*}
holds whenever $(g_j)_{j \in J}$ are orthonormal functions in $L^2(\R)$ and $(\nu_j)_{j \in J}$ are complex coefficients. Here the indexing set $J$ can be taken to be finite or countable.
\end{conjecture}

This conjecture is known for $n=2$ due to Frank--Sabin~\cite{FS2017}, and was shown to hold up to endpoints (that is, for $q > q_3$ and $r < q/6$) when $n = 3$ in \cite[Theorem 1.1]{CDH}. For $n \geq 4$, partial results were established in \cite[Theorem 1.1]{CDH}, but these fall short of the range of estimates stated in the conjecture. See \cite{CDH} for further details, including a discussion of the difficulties in the $n \geq 4$ case.

Our second result proves a special case of Conjecture~\ref{conj: CDH} up to the conjectured endpoint, where the system $(g_j)_{j \in J}$ is comprised of \textit{direction-separated wave packets} (see Definition~\ref{dfn: wp} and Definition~\ref{dfn: dir sep} below).

\begin{theorem}\label{thm: square fn} Let $n \geq 2$ and $\gamma \colon [-2,2] \to \R^n$ be a smooth, non-degenerate curve. For $q > q_n$ and $r \leq q/2n$, the inequality
\begin{equation}\label{eq: ON Fourier rest}
    \left\|\sum_{T \in \bdT}\nu_T |\cE_{\gamma} g_T|^2 \right\|_{L^{q/2}(\R^n)} \lesssim \left( \sum_{T \in \bdT} |\nu_T|^r \right)^{1/r}
\end{equation}
holds whenever $(g_T)_{T \in \bdT}$ are $L^2$-normalised wavepackets associated to direction-separated plates and $(\nu_T)_{T \in \bdT}$ are complex coefficients.
\end{theorem}

Here we say $g \in L^2(\R)$ is \textit{$L^2$-normalised} if $\|g\|_{L^2(\R)} = 1$. We refer the reader to Definition~\ref{dfn: wp} for the definition of a \textit{wave packet} in this context, and Definition~\ref{dfn: dir sep} for the definition of \textit{direction-separated}. Wave packets associated to direction-separated plates provide a natural class of orthonormal systems and, indeed, Conjecture~\ref{conj: CDH} is motivated by testing \eqref{eq: ON Fourier rest} against systems of such wave packets (see \cite[Proposition 4.1]{CDH}). Thus, Theorem~\ref{thm: square fn} arguably provides strong evidence for the validity of Conjecture~\ref{conj: CDH}, at least ruling out a large class of potential counterexamples. 

To prove Theorem~\ref{thm: square fn}, we again apply the Arkhipov--Chubarikov--Karatsuba estimate \eqref{eq: ACK intro}. In particular, \eqref{eq: ACK intro} is used to pointwise bound the functions $\cE_{\gamma}g_T$ and thereby reduce \eqref{eq: ON Fourier rest} to a purely geometric Kakeya-type maximal inequality. The reduction described here is somewhat in the spirit of classical arguments from \cite{Bourgain1991} in the hypersurface case, although the stationary phase machinery used in the space curve case is somewhat more involved. Moreover, the functions $\cE_{\gamma} g_T$ do not have simple localisation properties as in the hypersurface case. 

A multilinear variant of the key Kakeya inequality follows from very elementary arguments. The main task is then to reduce the \textit{bona fide} Kakeya inequality to this multilinear variant. This is achieved by appeal to a simple formulation of the Bourgain--Guth~\cite{BG2011} \textit{broad/narrow} method. It is remarked that the approach used here is rather similar to arguments used by Ham--Lee~\cite{HL2014} to prove weighted versions of Drury's~\cite{Drury1985} classical restriction theorem. 

%%%%%%%%%%%%%%%%%%%%%%%%%%%%%%%%%%%%%%%%%%%%%%%%%%%%%%%%%%%%%%%%%%%%%%%%%%%%%%%%%%%%%%%%%%%%%%%%

%   Notational conventions

%%%%%%%%%%%%%%%%%%%%%%%%%%%%%%%%%%%%%%%%%%%%%%%%%%%%%%%%%%%%%%%%%%%%%%%%%%%%%%%%%%%%%%%%%%%%%%%%

%\subsection*{Notational conventions} Given a list of objects $L$ and real numbers $A$, $B \geq 0$, we use the notation $A \lesssim_L B$, $B \gtrsim_L A$, $A = O_L(B)$ or $B =\Omega_L(A)$ to all indicate $A \leq C_L B$ for some constant $C_L$ which depends only on items in the list $L$ and possibly a choice of dimension $n$. We write $A \sim_L B$ to indicate $A \lesssim_L B$ and $B \lesssim_L A$.

\subsection*{Acknowledgements} The first author is supported by New Investigator Award EP/X011488/1. The second and third authors are supported by New Investigator Award UKRI097.

\subsection*{Tool and computational resource disclosure} No AI tools were used in the production of this paper.

%%%%%%%%%%%%%%%%%%%%%%%%%%%%%%%%%%%%%%%%%%%%%%%%%%%%%%%%%%%%%%%%%%%%%%%%%%%%%%%%%%%%%%%%%%%%%%%%

%   $L^2$ Fourier restriction via Arkhipov--Chubarikov--Karatsuba

%%%%%%%%%%%%%%%%%%%%%%%%%%%%%%%%%%%%%%%%%%%%%%%%%%%%%%%%%%%%%%%%%%%%%%%%%%%%%%%%%%%%%%%%%%%%%%%%

\section{\texorpdfstring{$L^2$ Fourier restriction via Arkhipov--Chubarikov--Karatsuba estimates}{L2 Fourier restriction via Arkhipov-Chubarikov-Karatsuba estimates}}\label{sec: rest via ACK}

The goal of this section is to prove Theorem~\ref{thm: L2 endpoint}, following the oscillatory integral scheme discussed in the introduction. Recalling the definition $q_n \coloneq n(n+1)$ from \eqref{eq: crit q}, the main step is to prove the following restricted weak-type inequality.  

\begin{proposition}\label{prop: weak endpoint} For a smooth, non-degenerate curve $\gamma \colon [2,2] \to \R^n$, we have
\begin{equation}\label{eq: weak endpoint}
  \|\cE_{\gamma}\cE_{\gamma}^* g\|_{L^{q_n,\infty}(\R^n)} \lesssim_{\gamma} \|g\|_{L^{q_n',1}(\R)}.   
\end{equation}
\end{proposition}

Assuming \eqref{eq: weak endpoint}, by the Lorentz space H\"older inequality,
\begin{equation*}
    \|\cE_{\gamma}^* g\|_{L^2(\R)}^2 = \inn{\cE_{\gamma} \cE_{\gamma}^* g}{g} \leq \|\cE_{\gamma}\cE_{\gamma}^* g\|_{L^{q_n, \infty}(\R^n)}\|g\|_{L^{q_n', 1}(\R^n)} \lesssim_{\gamma} \|g\|_{L^{q_n',1}(\R^n)}^2. 
\end{equation*}
Thus, by duality, we obtain the weak-type version of Theorem~\ref{thm: L2 endpoint}. Moreover, once Proposition~\ref{prop: weak endpoint} is established, we shall be able to upgrade the above estimate to the strong-type inequality of Theorem~\ref{thm: L2 endpoint} via a variant of the method of Bak--Seeger \cite{BS2011}. We remark that this upgrading procedure also exploits various ingredients introduced in the proof of Proposition~\ref{prop: weak endpoint}.

%%%%%%%%%%%%%%%%%%%%%%%%%%%%%%%%%%%%%%%%%%%%%%%%%%%%%%%%%%%%%%%%%%%%%%%%%%%%%%%%%%%%%%%%%%%%%%%%

%   Preliminaries

%%%%%%%%%%%%%%%%%%%%%%%%%%%%%%%%%%%%%%%%%%%%%%%%%%%%%%%%%%%%%%%%%%%%%%%%%%%%%%%%%%%%%%%%%%%%%%%%

\subsection{Preliminaries} \label{sub:preliminaries}

We discuss preliminaries for the proof of Proposition~\ref{prop: weak endpoint}.\medskip

\noindent \textit{Convolution kernels.} Let $a \in C^\infty(\R^n \times \R)$ be bounded and with $\supp{t}a \subseteq [-1,1]$. Define an operator by setting
\[ U[a]f \coloneq K[a] * f \quad\text{where}\quad K[a](x) \coloneq \int_{\R} e^{2 \pi i \gamma(t) \cdot x} a(x;t)\; \ud t. \]
Then $\cE_\gamma \cE_\gamma^* = U[|\alpha|^2]$, where $\alpha$ is the cutoff appearing in \eqref{eq:2eio1j2oiej12oilabel}, which depends only on $t$. To simplify notation, let $U_\gamma \coloneq \cE_\gamma \cE_\gamma^*$ and $K_\gamma \coloneq K[|\alpha|^2]$.\medskip

\noindent \textit{Quantified non-degeneracy.} Throughout what follows, we shall work with a fixed smooth, non-degenerate curve $\gamma \colon [-2, 2] \to \R^n$. The non-degeneracy condition \eqref{eq: nondegenerate}, together with a compactness argument, implies that there exists some $B \geq 1$ with $B = O_{\gamma}(1)$ such that
\begin{equation}\label{eq: non-deg 1}
 |\gamma^{(r)}(t)| \geq B^{-1} \quad \textrm{for $1 \leq r \leq n$}  \quad \textrm{and} \quad \max_{1 \leq r \leq n} |\gamma^{(r)}(t) \cdot x| \geq B^{-1}|x|
\end{equation}
and
\begin{equation}\label{eq: non-deg 2}
     |\gamma^{(r)}(t)| \leq B \qquad \textrm{for $1 \leq r \leq 2n$}
\end{equation}
for all $t \in [-2,2]$ and all $x \in \R^n$. We let $\fkG(B)$ denote the class of all curves $\gamma: [-2,2] \to \R^n$ satisfying \eqref{eq: non-deg 1} and \eqref{eq: non-deg 2}. We remark that the implicit constant in Proposition \ref{prop: weak endpoint} is uniform over all curves in a fixed class $\fkG(B)$, which will prove useful in finishing the proof of Theorem~\ref{thm: L2 endpoint} in \S\ref{sub:strong_type_upgrade_the_bak_seeger_argument}.

%%%%%%%%%%%%%%%%%%%%%%%%%%%%%%%%%%%%%%%%%%%%%%%%%%%%%%%%%%%%%%%%%%%%%%%%%%%%%%%%%%%%%%%%%%%%%%%%

%   Initial decomposition

%%%%%%%%%%%%%%%%%%%%%%%%%%%%%%%%%%%%%%%%%%%%%%%%%%%%%%%%%%%%%%%%%%%%%%%%%%%%%%%%%%%%%%%%%%%%%%%%

\subsection{Initial decomposition}

Let $C \geq 1$, chosen sufficiently large so as to satisfy the requirements of the forthcoming argument, depending only on $B$ and $n$. Fix $\zeta \in C^{\infty}_c(\R)$ satisfying
\begin{equation*}
    0 \leq \zeta \leq 1, \qquad \supp{} \zeta \subseteq [-1,1] \qquad \textrm{and} \qquad \sum_{\mu \in \Z} \zeta(\,\cdot\, - \mu) \equiv 1. 
\end{equation*}
Given $\lambda \geq 1$, we let $\fkM(\lambda) \coloneq \Z \cap [-2C\lambda, 2C\lambda]$. For all $\mu \in \fkM(\lambda)$, define $t_{\mu} \coloneq (C\lambda)^{-1}\mu \in [-2,2]$ and $\zeta_{\mu}(t) \coloneq \zeta(C\lambda t-\mu)$, so that 
\begin{equation*}%\label{eq: zeta mu}
    \supp{} \zeta_{\mu} \subseteq I_{\mu} \coloneq [t_{\mu} - (C\lambda)^{-1}, t_{\mu} + (C\lambda)^{-1}].
\end{equation*} 
Then $1 = \sum_{\mu \in \Z} \zeta_\mu(t)$ for $t \in [-1,1]$.

Let $\eta_{\circ} \in C^{\infty}_c(\R)$ be an even, radially non-increasing function, satisfying
\begin{equation*}
 0 \leq\eta_{\circ} \leq 1, \qquad \supp{} \eta_{\circ} \subseteq [-2,2] \qquad \textrm{and} \qquad\eta_{\circ}(s) = 1 \quad \textrm{for} \quad s \in [-1,1].   
\end{equation*}
For $1 \leq r \leq n$ and $\mu \in \fkM(\lambda)$, define
\[ \eta_{\mu,r}(x) \coloneq \eta_{\circ} \big( \gamma^{(r)}(t_{\mu})\cdot x \big). \]
For $1 \leq \ell \leq n$, and each $\mu \in \fkM(\lambda)$, define
\begin{equation} \label{eq: eta mu}
    \eta_\mu^{\ell}(x) \coloneq \big(1 - \eta_{\mu,\ell} ( C^{-1}\lambda^{-\ell} x ) \big)\prod_{r = 1}^{\ell - 1} \eta_{\mu,r}\left( C^{-1} \lambda^{-r} x \right) .
\end{equation}
We interpret an empty product as equal to $1$ in what follows. Thus \eqref{eq: eta mu} defines the function $\eta_{\mu}^1(x) \coloneq (1 - \eta_{\mu,1}(C^{-1}\lambda^{-1}x))$. Also, for $\mu \in \fkM(\lambda)$, define
\begin{equation*}
  \eta_\mu(x) \coloneq \prod_{r = 1}^n \eta_{\mu,r} \left( C^{-1} \lambda^{-r} x \right) \quad \textrm{so that} \quad 1 = \eta_\mu(x) + \sum_{\ell = 1}^n \eta_\mu^{\ell}(x) \quad \textrm{for all $x \in \R^n$.}
\end{equation*} 

Given a smooth bump function $\alpha$, and $\lambda > 0$, define
\begin{equation}\label{eq: a high}
  a_\high^\lambda(x;t) \coloneq \sum_{\mu \in \fkM(\lambda)} \sum_{\ell = 1}^n a_\mu^\ell(x;t) \quad\text{where}\quad a_\mu^\ell(x;t) \coloneq  \eta_\mu^{\ell}(x)\zeta_\mu(t)|\alpha(t)|^2    
\end{equation} 
and
\begin{equation*}
  a_{\low}^\lambda(x;t) \coloneq \sum_{\mu \in \fkM(\lambda)} a_\mu(x;t) \quad\text{where}\quad a_\mu(x;t) \coloneq    \eta_\mu(x) \zeta_\mu(t) |\alpha(t)|^2.   
\end{equation*} 
Thus, $|\alpha(t)|^2 = a_\high^\lambda(x;t) + a_{\low}^\lambda(x;t)$ for $t \in [-2,2]$, $x \in \R^n$, and $\lambda \geq 1$.

For each $\lambda \geq 1$, let $K^\lambda_\low \coloneq K[a^\lambda_\low]$ and $K^\lambda_\high \coloneq K[a^\lambda_\high]$ denote the corresponding kernels to the symbols defined above. Then
\[ K_\gamma = K^\lambda_\low + K^\lambda_\high. \]
Also define the associated convolution operators
\begin{equation*}
    U_{\low}^{\lambda}  \coloneq U[a^\lambda_\low] \quad\text{and}\quad U_{\high}^{\lambda} \coloneq U[a^\lambda_\high],
\end{equation*}
so that $U_\gamma = U_{\low}^{\lambda} + U_{\high}^{\lambda}$. Our goal is to prove $L^p$ bounds for $U_{\low}^{\lambda}$ and $U_{\high}^{\lambda}$, tracking the dependence of our estimates on the $\lambda$ parameter.

%%%%%%%%%%%%%%%%%%%%%%%%%%%%%%%%%%%%%%%%%%%%%%%%%%%%%%%%%%%%%%%%%%%%%%%%%%%%%%%%%%%%%%%%%%%%%%%%

%                                          Estimating the low part

%%%%%%%%%%%%%%%%%%%%%%%%%%%%%%%%%%%%%%%%%%%%%%%%%%%%%%%%%%%%%%%%%%%%%%%%%%%%%%%%%%%%%%%%%%%%%%%%

\subsection{Estimating the low part} Using the Fourier support properties of the kernel $K_{\low}^{\lambda}[a]$, we establish the following $L^2$ bound for $U^\lambda_\low$.

\begin{proposition}\label{prop: low} For all $\lambda \geq 1$ we have
\begin{equation*}
    \| \widehat{K}^\lambda_{\low} \|_{L^\infty(\R^n)} \lesssim_B \lambda^{\frac{n(n+1)}{2} - 1} \quad\text{and so}\quad \|U_{\low}^{\lambda} \|_{L^2(\R^n) \to L^2(\R^n)} \lesssim_B \lambda^{\frac{n(n+1)}{2} - 1}.
\end{equation*}
\end{proposition}

As a preliminary to the proof, for each $\mu \in \fkM(\lambda)$, let $K_{\mu}^{\lambda} \coloneq K[a_{\mu}]$. Then
\[ K^\lambda_{\low} = \sum_{\mu \in \fkM(\lambda)} K^\lambda_\mu. \]
Note that
\begin{equation}\label{eq: low 1}
  \widehat{K}_{\mu}^{\lambda}(\xi) = \int_{\R} \int_{\R^n} e^{-2\pi i x \cdot (\xi - \gamma(s))} \eta_{\mu}(x) \,\ud x\; \zeta_{\mu}(s)|\alpha(s)|^2\,\ud s.
\end{equation}
We shall make use of this formula in the proof of Proposition~\ref{prop: low} and also later. 

\begin{proof}[Proof (of Proposition~\ref{prop: low})] 
It is not difficult to see that 
\begin{equation*}
  \| \partial_x^{\beta} \eta_{\mu} \|_{L^\infty(\R^n)} \lesssim_{\beta} \lambda^\beta \qquad \textrm{for all $\beta \in \N_0^n$.}
\end{equation*}
Furthermore, the support of $\eta_{\mu}$ (and therefore all of its derivatives) is contained in
\begin{equation*}
    \bigcap_{r = 1}^n \big\{ x \in \R^n : |\gamma^{(r)}(t_{\mu}) \cdot x| \leq 2C\lambda^r\big\}.
\end{equation*}
This is in turn contained in a parallelepiped with dimensions
\begin{equation*}
  O(\lambda) \times O(\lambda^2) \times \cdots \times O(\lambda^n),  
\end{equation*}
from which we conclude the support of $\eta_\mu$ has volume bounded by $\lambda^{\frac{n(n+1)}{2}}$.

Thus, by applying a change of variable and repeated integration-by-parts to the $x$-integral in \eqref{eq: low 1}, we see that
\begin{align}
\nonumber
   |\widehat{K}_{\mu}^{\lambda}(\xi)| &\lesssim_B \lambda^{\frac{n(n+1)}{2}} \int_{\R}(1 + \lambda|\xi - \gamma(s)|)^{-2}\, \zeta_{\mu}(s)\,\ud s\\
   \label{eq: low 2}
   &\lesssim \lambda^{\frac{n(n+1)}{2} - 1} \left( 1 + \lambda \inf_{s \in I_\mu} |\xi - \gamma(s)| \right)^{-2}.
\end{align}

 Let $\varepsilon \coloneq 2^{-1}B^{-2}$, where $B \geq 1$ is as in \eqref{eq: non-deg 1} and \eqref{eq: non-deg 2}. If $s$, $t \in [-2,2]$ satisfy $|s-t| \leq \varepsilon$, then Taylor's theorem yields
\begin{equation}\label{eq: low 2.5}
    |\gamma(s)-\gamma(t)|\geq \big(B^{-1} - B|s-t|\big)|s-t| \geq (2B)^{-1}|s-t|.
\end{equation}
We may assume without loss of generality that $\lambda^{-1} \leq \varepsilon/4$, since otherwise the desired result follows from trivial bounds. 

Let $\cJ$ denote a partition $[-2,2]$ into $O(\varepsilon^{-1})$ intervals $J$ of length $\varepsilon/4 < |J| \leq \varepsilon/2$. Define 
\begin{equation*}
   \fkM_J(\lambda) \coloneq \left\{ \nu \in \fkM(\lambda) : t_{\mu} \in J \right\} \qquad \textrm{for all $J \in \cJ$.} 
\end{equation*}
Given $J \in \cJ$, we let $\bar{J}$ denote the interval concentric to $J$ but with $|\bar{J}| = \varepsilon$. If $\mu \in \fkM_J(\lambda)$ and $t \in I_{\mu}$, then it follows that $t \in \bar{J}$.

Fix $J \in \cJ$ and suppose $\mu$, $\nu \in \fkM_J(\lambda)$ satisfy $|\mu - \nu| > 10$ and $s \in I_{\mu}$ and $t \in I_{\nu}$. Then $s$, $t \in \bar{J}$ so that $|s - t| \leq \varepsilon$ and so it follows from \eqref{eq: low 2.5} that
\begin{equation*}
 |\xi - \gamma(s)| + |\xi - \gamma(t)| \geq |\gamma(s) - \gamma(t)| \gtrsim_B |s-t| \gtrsim \lambda^{-1}|\mu - \nu|,
\end{equation*}
where the final inequality exploits the separation $|\mu - \nu| > 10$. Thus, there is a constant $0 < c \leq 1$, depending only on $B$ and $n$, such that for each $\xi \in \R^n$, either
\begin{equation}\label{eq: low 3}
    \lambda|\xi - \gamma(s)| \geq c |\mu - \nu| \qquad \textrm{or} \qquad \lambda|\xi - \gamma(t)| \geq c |\mu - \nu|.
\end{equation}

Given $\xi \in \R^n$, let $s(\xi) \in [-2,2]$ satisfy
\begin{equation*}
  |\xi - \gamma(s(\xi))| = \inf_{\mu \in \fkM_J(\lambda)}\inf_{s \in I_{\mu}} |\xi - \gamma(s)|.  
\end{equation*}
It then follows from \eqref{eq: low 2} that 
\begin{equation} \label{eq: low 4}
    |\widehat{K}_\mu^\lambda(\xi)| \lesssim_B \lambda^{\frac{n(n+1)}{2} - 1} \big(1 + \lambda|\xi - \gamma(s(\xi))|\big)^{-2} \qquad \textrm{for all $\mu \in \fkM_J(\lambda)$.}
\end{equation}

There exists some $\mu(\xi) \in \fkM_J(\lambda)$ such that $s(\xi) \in I_{\mu(\xi)}$. Decompose $\fkM_J(\lambda)$ as the disjoint union $\fkM_J(\lambda) = \fkM_J^0(\lambda; \xi) \cup \fkM_J^1(\lambda; \xi)$, where
\begin{align*}
  \fkM_J^0(\lambda; \xi) &\coloneq \left\{ \nu \in \fkM_J(\lambda) : \lambda |\xi - \gamma(s(\xi))| \geq  c|\mu(\xi) - \nu| \textrm{ or } |\mu(\xi) - \nu| \leq 10 \right\}, \\
  \fkM_J^1(\lambda; \xi) &\coloneq \fkM_J(\lambda) \setminus \fkM_J^0(\lambda; \xi). 
\end{align*}
If $\nu \in \fkM_J^1(\lambda; \xi)$, then $\lambda |\xi - \gamma(s(\xi))| \not\geq  c|\mu(\xi) - \nu|$ and $|\mu(\xi) - \nu| > 10$. Thus, applying \eqref{eq: low 3} with $s \coloneq s(\xi)$ and $t \in I_{\nu}$ arbitrary implies that $\lambda |\xi - \gamma(t)| \geq  c|\mu(\xi) - \nu|$ for all $t \in I_\nu$. Combining this with \eqref{eq: low 2} yields 
\begin{equation} \label{eq: low 5}
    |\widehat{K}_\nu^\lambda(\xi)| \lesssim_B \lambda^{\frac{n(n+1)}{2} - 1}\left( 1 + |\mu(\xi) - \nu|\right)^{- 2}. 
\end{equation}

Observe that $\#\fkM_J^0(\lambda; \xi) \lesssim_B 1 + \lambda |\xi - \gamma(s(\xi))|$ and so \eqref{eq: low 4} implies 
\begin{equation*}
   \sum_{\nu \in \fkM_J^0(\lambda;\xi)} |\widehat{K}^\lambda_\nu(\xi)|  \lesssim_B \lambda^{\frac{n(n+1)}{2} - 1} \left(1 + \lambda|\xi - \gamma(s(\xi))|\right)^{-2} \#\fkM_J^0(\lambda; \xi) \lesssim_B \lambda^{\frac{n(n+1)}{2} - 1}.
\end{equation*}
On the other hand, by \eqref{eq: low 5}, we have
\begin{equation*}
    \sum_{\nu \in \fkM_J^1(\lambda;\xi)} |\widehat{K}^\lambda_\nu(\xi)| \lesssim_B \lambda^{\frac{n(n+1)}{2} - 1} \sum_{\nu \in \fkM(\lambda)} \left( 1+ |\mu(\xi) - \nu|\right)^{-2} \lesssim \lambda^{\frac{n(n+1)}{2} - 1}.
\end{equation*}
Combining the two previous displays, summing over all $J \in \cJ$ and applying the triangle inequality, we conclude that $|\widehat{K}^\lambda_{\low}(\xi)| \lesssim_B \lambda^{\frac{n(n+1)}{2} - 1}$, as required. 
\end{proof}

%%%%%%%%%%%%%%%%%%%%%%%%%%%%%%%%%%%%%%%%%%%%%%%%%%%%%%%%%%%%%%%%%%%%%%%%%%%%%%%%%%%%%%%%%%%%%%%%

%                                          Estimating the high part

%%%%%%%%%%%%%%%%%%%%%%%%%%%%%%%%%%%%%%%%%%%%%%%%%%%%%%%%%%%%%%%%%%%%%%%%%%%%%%%%%%%%%%%%%%%%%%%%

\subsection{Estimating the high part} The second ingredient is an analogue of the dispersive estimate for the classical Stein--Tomas / Strichartz estimate. 

\begin{proposition}\label{prop: high} For all $\lambda \geq 1$ we have
\begin{equation*}    \|K_{\high}^{\lambda} \|_{L^{\infty}(\R^n)} \lesssim_B \lambda^{-1}.
\end{equation*}
\end{proposition}

Proposition~\ref{prop: high} relies on the classical argument used to prove Arkhipov--Chubarikov--Karatsuba oscillatory-type integral estimates \cite[p.13, Theorem 1.1]{ACK_book}, mentioned in the introduction. Recalling the definition in \eqref{eq: H functional} that
\[ H_{\gamma; J}(x) \coloneq \inf_{t \in J} \max_{1 \leq r \leq n} | x \cdot \gamma^{(r)}(t)|^{1/r}, \]
the following is a consequence of \cite[Theorem 3.1]{HS2025} (see also \cite{BGGIST2007}).

\begin{theorem} \label{thm: ACK}
    Let $\gamma \colon [-2,2] \to \R^n$ be a smooth, non-degenerate curve and $J \subseteq [-1,1]$ be an interval. If $b \in C_c^1(\R)$ satisfies $\supp{} b \subseteq J$, then
    \begin{equation*}%\label{eq: ACK}
        \bigg|\int_J e^{2 \pi i x \cdot \gamma(t)} b(t)\,\ud t\bigg| \lesssim \|b'\|_{L^1(\R)} H_{\gamma, J}(x)^{-1}\qquad \textrm{for all $x \in \R^n$.} 
    \end{equation*}
\end{theorem}

Indeed, Theorem~\ref{thm: ACK} follows from \cite[Theorem 3.1]{HS2025} by repeating the (standard) argument used to prove \cite[Corollary 3.2]{HS2025}. 

For the proof of Proposition~\ref{prop: high}, we do not use Theorem~\ref{thm: ACK} \textit{per se}, but rather the method of proof. This involves decomposing the domain into regions where classical van der Corput estimates (see, for example,  \cite[Chapter VIII, Proposition 2]{Stein_book}) can be applied. In our context, for the decomposition we rely on the following monotonicity property of the $\gamma^{(\ell)}(t)\cdot x$.

\begin{lemma}\label{lem: eta cutoff} There exists some $N \in \N$, depending only on $B$ and $n$, such that the following holds. Given $x \in \R^n$ and $I \subseteq [-1,1]$ an interval suppose that
\begin{equation*}%\label{eq: eta cutoff} 
    \inf_{t \in I}|\gamma^{(m)}(t)\cdot x| \geq c|x| \qquad \textrm{holds for some $c > 0$ and some $2 \leq m \leq n$.}
\end{equation*}
Then there exists a decomposition $\cI_x$ of $I$ into a family of intervals with the following properties:
\begin{enumerate}[i)]
    \item $\#\cI_x \leq N$;
    \item For $1 \leq r \leq m - 1$ and $I \in \cI_x$, the function $t \mapsto \gamma^{(r)}(t)\cdot x$ is monotone on $I$.
\end{enumerate}
\end{lemma}

Lemma~\ref{lem: eta cutoff} follows from the proof of \cite[Theorem 3.1]{HS2025}.   

\begin{proof}[Proof (of Proposition~\ref{prop: high})] We assume without loss of generality that $\lambda > 2B^2$, where $B \geq 1$ is as in \eqref{eq: non-deg 1} and \eqref{eq: non-deg 2}, since the result trivially holds for $1 \leq \lambda \leq 2B^2$.\medskip

\noindent \textit{Step 1: Further decomposition.} For $\ell \leq m \leq n$ and $\mu \in \fkM(\lambda)$, define
\begin{equation}\label{eq: high 1}
    a_\mu^{\ell,m}(x;t) \coloneq a_\mu^{\ell}(x;t)\; \eta_\mu^{\ell, m}(x), 
\end{equation} 
where
\begin{equation}\label{eq: high 2}
\eta_\mu^{\ell, m}(x) \coloneq \left( 1 - \eta_{\mu,m} \left( C |x|^{-1} x \right) \right) \prod_{r = \ell}^{m-1} \eta_{\mu,r} \left( C |x|^{-1} x \right) .
\end{equation}
For all $1 \leq \ell \leq n$ and all $x \in \supp{} \eta_\mu^{\ell}$, we claim that there exists some $\ell \leq m_0 \leq n$ for which $|\gamma^{(m_0)}(t_\mu) \cdot x| \geq B^{-1} |x|$. Provided $C \geq 2B$, it then follows that
\begin{equation}\label{eq: high 3}
  \sum_{m = \ell}^n a_\mu^{\ell,m}(x;t) = a_\mu^{\ell}(x;t) \left(1 - \prod_{r = \ell}^n  \eta_{\mu,r} \left( C |x|^{-1} x \right) \right) = a_\mu^{\ell}(x;t).   
\end{equation}
Indeed, here the second equality follows because $\eta_{\mu,m_0} (C |x|^{-1} x ) = 0$. 

To verify the above claim, first note that \eqref{eq: non-deg 1} guarantees that there exists some $1 \leq m_0 \leq n$ such that $|\gamma^{(m_0)}(t_\mu) \cdot x| \geq B^{-1} |x|$. We claim that $\ell \leq m_0 \leq n$. To see this, note that for $x \in \supp{} \eta_\mu^{\ell}$ and $1 \leq r < \ell$, we have
\begin{equation*}
    |\gamma^{(r)}(t_\mu) \cdot x| \leq 2 C \lambda^r \quad\text{and}\quad |x| \geq B^{-1} |\gamma^{(\ell)}(t_\mu) \cdot x| \geq B^{-1} C \lambda^\ell.
\end{equation*} 
Provided $\lambda > 2 B^2$, it follows that
\begin{equation*}
    B^{-1} |x| \geq B^{-2} C \lambda^\ell \geq (B^{-2} C \lambda^{\ell - r}) (2C)^{-1} |\gamma^{(r)}(t_\mu) \cdot x| > |\gamma^{(r)}(t_\mu) \cdot x|.
\end{equation*} 
This concludes the proof of the claim and therefore establishes \eqref{eq: high 3}.

Having performed the decomposition \eqref{eq: high 3} of $a_\mu^\ell$ into the functions $a_\mu^{\ell,m}$, define
\begin{equation*}
     a^{\ell,m} \coloneq \sum_{\mu \in \fkM(\lambda)} a^{\ell,m}_\mu \qquad \textrm{so that} \qquad   K^\lambda_\high = \sum_{\ell = 1}^n \sum_{m = \ell}^n K[a^{\ell,m}].
\end{equation*}
Now fix $1 \leq \ell \leq m \leq n$. In light of the above, it suffices to show
\begin{equation*}
    |K[a^{\ell,m}](x)| \lesssim_B \lambda^{-1} \qquad \textrm{for all $x \in \supp{x} a^{\ell,m}$.}
\end{equation*}

\noindent \textit{Step 2: Monotonicity properties.} Fix $x \in \supp{x} a^{\ell,m}$ so that $x \in \supp{x} a^{\ell, m}_{\mu}$ for some $\mu \in \fkM(\lambda)$. Using the definitions \eqref{eq: high 1} and \eqref{eq: high 2}, together with the properties of $\eta_{\circ}$, it follows that
\begin{equation*}
  |\gamma^{(m)}(t_{\mu}) \cdot x| \geq C^{-1} |x|,  
\end{equation*}
 where $C \geq 2B$. If $t \in [-2,2]$ satisfies $|t - t_{\mu}| \leq C^{-2}$, then $|\gamma^{(m)}(t) \cdot x| \geq (2C)^{-1} |x|$. This follows immediately from the mean value theorem, using the bound $|\gamma^{(m+1)}(s) \cdot x| \leq B |x|$ for any $s \in [-2,2]$ from \eqref{eq: non-deg 2}. In view of the above observation, we may cover $\supp{t} a^{\ell,m}$ by disjoint intervals $J_1,\dots,J_N$, where $N = O_{\gamma}(1)$, such that
\begin{equation}\label{eq: high 4}
|\gamma^{(m)}(t) \cdot x| \geq (2C)^{-1} |x| \quad\text{for}\quad t \in J_1 \cup \dots \cup J_N.     
\end{equation} 
Applying Lemma~\ref{lem: eta cutoff} to each $J_i$, and taking the union over $1 \leq i \leq N$ of the families of intervals obtained from the conclusion of this lemma, we can find a family of intervals $\cI_x$ covering $\supp{t} a^{\ell,m}$ such that $\# \cI_x = O_B(1)$ and, for each $I \in \cI_x$ and $1 \leq r < m$, the function $t \mapsto \gamma^{(r)}(t) \cdot x$ is monotone on $I$.

We may decompose $\R$ into five intervals upon which $\eta_\circ$ is either strictly monotone or constant. Thus, by splitting each interval in $\cI_x$ into a bounded number of disjoint subintervals, we can pass to a further cover, which we relabel as $\cI_x$, with the following properties:
\begin{enumerate}[(1)]
    \item $ \# \cI_x  = O_B(1)$;
    \item For all $I \in \cI_x$ and all $1 \leq r < m$, the sequence
\begin{equation*}
 \fkM_I(\lambda) \coloneq \{ \eta_{\mu,r}(x) : \mu \in \fkM(\lambda)\ \text{and}\ t_\mu \in I \}    
\end{equation*} 
is monotone as a function of $\mu$.
\end{enumerate}
Furthermore, we claim that we can pass to a further cover so that the following additional property also holds:
\begin{enumerate}[(1)]\addtocounter{enumi}{2}
\item For all $I \in \cI_x$, either $|I| \leq 2\lambda^{-1}$ or 
\begin{equation} \label{eq: high 5}
    |\gamma^{(\ell)}(t) \cdot x| \geq \lambda^{\ell}/2 \qquad \textrm{for all $t \in I$.}
\end{equation}
\end{enumerate}
To ensure \eqref{eq: high 5} holds, we consider two cases.\medskip

\noindent \textit{Case 1: $\ell = m$.} Since $I$ is contained in one of the intervals $J_1,\dots,J_N$ discussed at the beginning of this decomposition, we know from \eqref{eq: high 4} that $|\gamma^{(\ell)}(t) \cdot x| \geq (2C)^{-1} |x|$ for $t \in I$. The definitions \eqref{eq: eta mu}, \eqref{eq: a high} and \eqref{eq: high 1}, together with the properties of $\eta_{\circ}$, imply $|x| \geq |\gamma^{(\ell)}(t_{\mu}) \cdot x| \geq C\lambda^{\ell}$. Thus, $|\gamma^{(\ell)}(t) \cdot x| \geq \lambda^\ell / 2$. So the desired bound \eqref{eq: high 5} holds automatically, without further partitioning.\medskip

\noindent \textit{Case 2: $\ell < m \leq n$.} Fix $I \in \cI_x$, so that $t \mapsto \gamma^{(\ell)}(t) \cdot x$ is monotone on $I$. Thus, the set $G_I \coloneq \{t \in I : |\gamma^{(\ell)}(t) \cdot x| \geq \lambda^{\ell}\}$ is the union of at most two intervals, lying at the ends of $I$. If the remaining interval $B_I \coloneq I \setminus G_I$ satisfies $|B_I| \leq 2\lambda^{-1}$, then we are done. Otherwise, we may decompose $B_I = B_I^- \cup S_I \cup B_I^+$ where $S_I$, $B_I^-$ and $B_I^+$ are all intervals, which satisfy $|B_I^-| = |B_I^+| = \lambda^{-1}$ and $\dist(S_I, G_I) \geq \lambda^{-1}$. If $|t - t_\mu| \leq \lambda^{-1}$ for some $\mu \in \fkM(\lambda)$, then $t_\mu \not \in G_I$. It follows that $\eta^{\ell}_\mu(x) = 0$ and hence $a^{\ell,m}_\mu(x;t) = 0$. This means that we can drop the interval $S_I$ from the collection and still have a covering of $\supp{t} a^{\ell, m}$. Thus, we form the new covering by replacing each $I \in \cI_x$ with the constituent intervals of $G_I$ together with $B_I$ if $|B_I| \leq 2 \lambda^{-1}$, or with $B_I^-$ and $B_I^-$ if $|B_I| \geq 2 \lambda^{-1}$.\medskip
 
We define
\begin{equation*}
a^{\ell,m}_I \coloneq \sum_{\mu \in \fkM_I(\lambda)} a^{\ell,m}_\mu \qquad \textrm{for $I \in \cI_x$.}  
\end{equation*}
It then suffices to prove that $|K[a^{\ell,m}_I](x)| \lesssim_B \lambda^{-1}$ for each $I \in \cI_x$.\medskip

\noindent \textit{Step 3: Oscillatory integral estimates.} Fix an interval $I \in \cI_x$. If $|I| \leq 2\lambda^{-1}$, then we immediately have $|K[a^{\ell,m}_I](x)| \lesssim_B \lambda^{-1}$, owing to the size of the $t$-support of $a^{\ell,m}_I$. On the other hand, if $|I| > 2\lambda^{-1}$, then it follows from property (3) above that \eqref{eq: high 5} holds. This implies that $H_{\gamma,I}(x) \geq \lambda$. 

Observe that
\begin{equation*}
    K[a^{\ell,m}_I](x) = \int_I e^{2 \pi i x \cdot \gamma(t)} b_I(t)\; \ud t \qquad \textrm{where} \qquad b_I(t) \coloneq \sum_{\mu \in \fkM_I(\lambda)} a^{\ell, m}_\mu(x;t).
\end{equation*}
Suppose we knew
\begin{equation}\label{eq: high 6}
   \| b_I' \|_{L^1([-2,2])} \lesssim_B 1.
\end{equation}
The required bound $|K[a^{\ell,m}_I](x)| \lesssim_B \lambda^{-1}$ then follows from Theorem~\ref{thm: ACK}. In fact, we can simply apply the classical van der Corput estimates (see, for example,  \cite[Chapter VIII, Proposition 2]{Stein_book}) here: the above decomposition essentially mirrors the argument used to reduce the proof of  Theorem~\ref{thm: ACK} to van der Corput's lemma.\medskip

\noindent \textit{Step 4: Bounding the amplitude.} It remains to prove \eqref{eq: high 6}. For this, we may assume $\fkM_I(\lambda)$ is nonempty. Write $\bar{\lambda} \coloneq C \lambda$. By considering the extrema of $\fkM_I(\lambda)$, we can find $\mu_0$, $\mu_1 \in \fkM_I(\lambda)$ such that $\supp{} b_I \subseteq [t_0 - \bar{\lambda}^{-1}, t_1 + \bar{\lambda}^{-1}]$ where $t_i \coloneq t_{\mu_i}$ for $i = 0$, $1$. For any interval $J \subseteq \R$ we have
\begin{equation*}
    \| b_I' \|_{L^1(J)} \leq |J| \| b_I' \|_{L^{\infty}(\R)} \lesssim_B |J|\bar{\lambda}.
\end{equation*}  
Applying this with $J = [t_0 - \bar{\lambda}^{-1}, t_0]$ and $J = [t_1, t_1 +  \bar{\lambda}^{-1}]$, we conclude $\| b_I' \|_{L^1(J)} \lesssim 1$ in either case. Matters are thus reduced to proving
\begin{equation}\label{eq: high 7}
    \| b_I' \|_{L^1([t_0,t_1])} \lesssim_B 1.
\end{equation}

A change of variables allows us to write 
\begin{align*}
    \| b_I' \|_{L^1([t_0,t_1])} = \int_0^{\bar{\lambda}^{-1}} \sum_{\mu = \mu_0}^{\mu_1 - 1} |b_I'(t_\mu + t)|\; \ud t.
\end{align*}
For each $\mu_0 \leq \mu \leq \mu_1$ and $t \in [0,\bar{\lambda}^{-1}]$, note that
\begin{equation*}
    b_I(t_{\mu} + t) = \alpha(t_{\mu} + t) \left(\zeta_\mu(t_{\mu} + t) \eta_\mu^{\ell}(x) \eta_\mu^{\ell, m}(x) + \zeta_{\mu+1}(t_{\mu} + t) \eta_{\mu+1}^{\ell}(x) \eta_{\mu+1}^{\ell, m}(x)\right). 
\end{equation*} 
We note also that $\zeta_{\mu}(t_{\mu} + t) + \zeta_{\mu+1}(t_{\mu} + t) = 1$ for $t \in [0,\bar{\lambda}^{-1}]$. Differentiating, we find that $\zeta_{\mu + 1}'(t_{\mu} + t) = - \zeta_{\mu}'(t_{\mu} + t)$ for $t \in (0,\bar{\lambda}^{-1})$. Thus,
\begin{equation}\label{eq: high 8}
     b_I'(t_{\mu} + t) = \zeta_{\mu+1}'(t_{\mu} + t) \left( \eta_{\mu+1}^{\ell}(x) \eta_{\mu+1}^{\ell, m}(x) - \eta_\mu^{\ell}(x) \eta_\mu^{\ell, m}(x) \right) + O\left( 1 \right). 
\end{equation}
Inserting the bound $|\zeta_{\mu+1}'(t_{\mu} + t)| \lesssim \bar{\lambda}$ and integrating, we conclude that
\begin{equation}\label{eq: high 9}
    \| b_I' \|_{L^1([t_0,t_1])} \lesssim_B 1 + \sum_{\mu = \mu_0}^{\mu_1 - 1} \left| \eta_{\mu+1}^{\ell}(x) \eta_{\mu+1}^{\ell, m}(x) - \eta_\mu^{\ell}(x) \eta_\mu^{\ell, m}(x) \right|. 
\end{equation} 
Here we have used the bound $\mu_1 - \mu_0 \lesssim \bar{\lambda}$ to control the contribution from the $O(1)$ term in \eqref{eq: high 8}. Thus, to complete the proof it suffices to prove the sum on the right-hand side of \eqref{eq: high 9} is $O(1)$.

For $\ell < m \leq n$, define
\begin{align*}
    A_\mu &\coloneq (1 - \eta_{\mu,\ell}( C^{-1} \lambda^{-\ell} x )) \prod_{r = 1}^{\ell - 1} \eta_{\mu,r}( C^{-1} \lambda^{-r} x) \prod_{r = \ell}^{m - 1} \eta_{\mu,r}( C |x|^{-1} x ),\\
    B_\mu &\coloneq 1 - \eta_{\mu,m}( C |x|^{-1} x ) ,
\end{align*}
and for $\ell = m$ define
\begin{align*}
    A_\mu &\coloneq \prod_{r = 1}^{\ell - 1} \eta_{\mu,r}( C^{-1} \lambda^{-r} x),\\
    B_\mu &\coloneq (1 - \eta_{\mu,m}( C |x|^{-1} x )) (1 - \eta_{\mu,m}( C^{-1} \lambda^{-\ell} x )),
\end{align*}
In either case, $\eta_\mu^{\ell}(x) \eta_\mu^{\ell, m}(x) = A_\mu B_\mu$. Note that we may write
\begin{equation}\label{eq: high 10}
   A_\mu = \prod_{i = 1}^N a_{i, \mu} 
\end{equation}
for some $1 \leq N \leq n$ and monotone real sequences $( a_{i,\mu})_{\mu = \mu_0}^{\mu_1}$ taking values in $[0,1]$ for $1 \leq i \leq N$ (indeed, these monotone quantities are precisely the terms in the definition of $A_\mu$ in either case). On the other hand, we claim that for all $\mu \in \fkM_I(\lambda)$, we have $|B_{\mu+1} - B_\mu| \lesssim_B \lambda^{-1}$. This would follow from showing that for all $\mu \in \fkM_I(\lambda)$, and all $R > 0$,
\[ |\eta_{\mu+1,m}( R^{-1} x ) - \eta_{\mu,m}( R^{-1} x )| \lesssim_B \lambda^{-1}. \]
We note that $|\gamma^{(m)}(t_\mu) \cdot x| \geq C |x|$ and $|\gamma^{(m)}(t_{\mu+1}) \cdot x| \geq C |x|$ for all $\mu \in \fkM_I(\lambda)$. So if $R \leq |x|/2$, then $\eta_{\mu,m}( R^{-1} x) = \eta_{\mu+1,m} ( R^{-1} x ) = 0$, and the claim is trivial. On the other hand, if $R \geq |x| / 2$, then using the mean-value theorem, together with \eqref{eq: non-deg 2}, we find
\begin{align*}
    |\eta_{\mu+1,m}( R^{-1} x ) - \eta_{\mu,m}( R^{-1} x )| &\lesssim R^{-1} |(\gamma^{(m)}(t_{\mu+1}) - \gamma^{(m)}(t_\mu)) \cdot x|\\
    &\leq R^{-1}\; B|x| \lambda^{-1}\\
    &\leq 2B \lambda^{-1} \lesssim_B \lambda^{-1}.
\end{align*}
So the bound continues to hold in this case.

Combining the bound $|B_{\mu+1} - B_\mu| \lesssim \lambda^{-1}$ with the bounds $|A_{\mu+1}| \leq 1$ and $|B_\mu| \leq 1$, we deduce that
\begin{align}
\nonumber
    \sum_{\mu = \mu_0}^{\mu_1 - 1} \left| \eta_{\mu+1}^{\ell}(x) \eta_{\mu+1}^{\ell, m}(x) - \eta_\mu^{\ell}(x) \eta_\mu^{\ell, m}(x) \right|
    &\leq \sum_{\mu = \mu_0}^{\mu_1 - 1} |A_{\mu+1}| |B_{\mu+1} - B_\mu| + |A_{\mu+1} - A_\mu| |B_\mu|\\
    \label{eq: high 11}
    &\lesssim_B 1 + \sum_{\mu = \mu_0}^{\mu_1 - 1} |A_{\mu+1} - A_\mu|.
\end{align}

Using the formula in \eqref{eq: high 10}, the discrete Leibnitz rule and the triangle inequality,
    \begin{align}
        \sum_{\mu = \mu_0}^{\mu_1 - 1} |A_{\mu + 1} - A_\mu| &\leq \sum_{i = 1}^N \sum_{\mu = \mu_0}^{\mu_1 - 1} |a_{i,\mu+1} - a_{i,\mu}|  \prod_{j = 1}^{i-1} |a_{j,\mu}| \prod_{j = i+1}^N |a_{j,\mu+1}|\nonumber\\
        &\leq \sum_{i = 1}^N \sum_{\mu = \mu_0}^{\mu_1 - 1} |a_{i,\mu+1} - a_{i,\mu}|. \label{eq: high 12}
    \end{align}
    The monotonicity of the sequence $( a_{i,\mu})_{\mu = \mu_0}^{\mu_1}$ implies
    \begin{equation} \label{eq: high 13}
        \sum_{\mu = \mu_0}^{\mu_1 - 1} |a_{i,\mu+1} - a_{i,\mu}| = |a_{i,\mu_1} - a_{i,\mu_0}| \leq 2.
    \end{equation}
    Combining \eqref{eq: high 11}, \eqref{eq: high 12}, \eqref{eq: high 13} with \eqref{eq: high 9} completes the proof of \eqref{eq: high 7}.
\end{proof}

%%%%%%%%%%%%%%%%%%%%%%%%%%%%%%%%%%%%%%%%%%%%%%%%%%%%%%%%%%%%%%%%%%%%%%%%%%%%%%%%%%%%%%%%%%%%%%%%

%                                          Proof of the weak-type bound: combining the estimates

%%%%%%%%%%%%%%%%%%%%%%%%%%%%%%%%%%%%%%%%%%%%%%%%%%%%%%%%%%%%%%%%%%%%%%%%%%%%%%%%%%%%%%%%%%%%%%%%

\subsection{Proof of the weak-type bound: combining the estimates}\label{subsec: combining} We now combine the results of the previous subsections to prove Proposition~\ref{prop: weak endpoint}. 

\begin{proof}[Proof (of Proposition~\ref{prop: weak endpoint})] Via the standard classification of Lorentz norm bounds via layer cake decompositions (see, for instance, \cite[Chapter V, Theorem 3.13]{MR304972}), prove \eqref{eq: weak endpoint} we may assume without loss of generality that $g = \chi_E$ is a characteristic function of a measurable set $E \subset \R^n$ of finite measure. It then suffices to show
\begin{equation}\label{eq: Lp combine 2}
    \inn{|U_\gamma \chi_E|}{\chi_F} \lesssim_{B, a} |E|^{1 - \frac{1}{n(n+1)}} |F|^{1 - \frac{1}{n(n+1)}}
\end{equation}
holds whenever $a \in \fkA^{\ell}(\gamma)$ for some $1 \leq \ell \leq n$ and $E$, $F \subseteq \R^n$ are measurable sets of finite measure. Indeed, \eqref{eq: Lp combine 2} implies \eqref{eq: weak endpoint} for $g = \chi_E$ by taking $F$ to be a superlevel set of $|U_\gamma \chi_E|$. 

Turning to the proof of \eqref{eq: Lp combine 2}, we first consider the case $|E||F| \geq 1$; as shown below, the remaining case can be dealt with using trivial estimates.

For any choice of $\lambda \geq 1$, we can split the operator into two parts, \textit{viz.}
\begin{equation*}
     \inn{|U_\gamma \chi_E|}{\chi_F} \leq  \inn{|U_{\low}^{\lambda} \chi_E|}{\chi_F} +  \inn{|U_{\high}^{\lambda} \chi_E|}{\chi_F}.
\end{equation*}
By Proposition~\ref{prop: low} and the Cauchy--Schwarz inequality, 
\begin{align}
\inn{|U_{\low}^{\lambda} \chi_E|}{\chi_F}| &\leq \|K_{\low}^{\lambda} \ast \chi_E\|_{L^2(\R^n)} |F|^{1/2}
\label{eq: Lp combine 3} \lesssim_B \lambda^{\frac{n(n+1)}{2} - 1} |E|^{1/2}|F|^{1/2}.
\end{align}
On the other hand, by Proposition~\ref{prop: high} and Young's convolution inequality, we have
\begin{equation}\label{eq: Lp combine 4}
    \inn{|U_{\high}^{\lambda} \chi_E|}{\chi_F}| \leq \|K_{\high}^{\lambda} \ast \chi_E\|_{L^{\infty}(\R^n)} |F| \lesssim_B \lambda^{-1} |E||F|.
\end{equation}
If we take $\lambda = |E|^{\frac{1}{n(n+1)}}|F|^{\frac{1}{n(n+1)}} \geq 1$, then the right-hand sides of both \eqref{eq: Lp combine 3} and \eqref{eq: Lp combine 4} agree. From this, we deduce that 
\begin{equation*}
 \inn{|U_\gamma \chi_E|}{\chi_F}| \lesssim_{B, a} |E|^{1 - \frac{1}{n(n+1)}}|F|^{1 - \frac{1}{n(n+1)}},
\end{equation*}
as required. 

It remains to consider the case  $|E||F| \leq 1$. Trivially, we have
\begin{equation*}
   \inn{|U_\gamma \chi_E|}{\chi_F} \leq \| K_\gamma \ast \chi_E \|_{L^{\infty}(\R^n)}|F| \leq \| K_\gamma \|_{L^{\infty}(\R^n)} |E||F|.
\end{equation*}
It is clear from the definitions that $\| K_\gamma \|_{L^\infty(\R^n)} \lesssim_B 1$, and so
\begin{equation*}
   \inn{|U_\gamma \chi_E|}{\chi_F} \lesssim_{B, a} |E||F| \leq  |E|^{1 - \frac{1}{n(n+1)}}|F|^{1 - \frac{1}{n(n+1)}},
\end{equation*}
where the last step uses the hypothesis $|E||F| \leq 1$.
\end{proof}

%%%%%%%%%%%%%%%%%%%%%%%%%%%%%%%%%%%%%%%%%%%%%%%%%%%%%%%%%%%%%%%%%%%%%%%%%%%%%%%%%%%%%%%%%%%%%%%%

%    Strong-type upgrade: the Bak--Seeger argument

%%%%%%%%%%%%%%%%%%%%%%%%%%%%%%%%%%%%%%%%%%%%%%%%%%%%%%%%%%%%%%%%%%%%%%%%%%%%%%%%%%%%%%%%%%%%%%%%

\subsection{Strong-type upgrade: the Bak--Seeger argument}\label{sub:strong_type_upgrade_the_bak_seeger_argument}

It remains to upgrade the weak-type inequality from Proposition~\ref{prop: weak endpoint} to the strong-type inequality from Theorem~\ref{thm: L2 endpoint}. As discussed above, this is achieved via a modification of an argument due to Bak--Seeger~\cite{BS2011}.

\begin{proof}[Proof (of Theorem~\ref{thm: L2 endpoint})] Recalling $U_{\gamma} \coloneq \cE_{\gamma}\cE_{\gamma}^*$, it suffices to show 
\begin{equation}\label{eq: B-S 1}
    \|U_\gamma f\|_{L^{q_n}(\R^n)}  \lesssim_{B, a} \|f\|_{L^{q_n'}(\R^n)}.
\end{equation}
Once \eqref{eq: B-S 1} is established, we can immediately deduce Theorem~\ref{thm: L2 endpoint} via a standard duality argument. Indeed, assuming \eqref{eq: B-S 1}, we calculate that
\begin{equation*}
    \|\cE_{\gamma}^* g\|_{L^2(\R^n)}^2 = \inn{\cE_{\gamma} \cE_{\gamma}^* g}{g} \leq \|\cE_{\gamma}\cE_{\gamma}^* g\|_{L^{q_n}(\R^n)}\|g\|_{L^{q_n', 1}(\R^n)} \lesssim_B \|g\|_{L^{q_n'}(\R^n)}^2. 
\end{equation*}
By duality, we conclude $\cE_\gamma$ is bounded from $L^2(\R^n)$ to $L^{q_n}(\R)$, as required.

Let $E$, $F \subseteq \R^n$ be measurable sets of finite measure. We shall show 
\begin{equation}\label{eq: B-S 2}
 |\inn{U_{\gamma}\chi_E}{\chi_F}| \lesssim_B |E|^{1/\rho_n}|F|^{1/\sigma_n'},
\end{equation}
where
\begin{equation*}
    \rho_n' \coloneq \frac{(n^2 + n + 2)n(n+1)}{4} \qquad \textrm{and} \qquad \sigma_n \coloneq \frac{n^2 + n + 2}{2}. 
\end{equation*}
Since the class of operators $U_{\gamma}$ is closed under taking adjoints, we also conclude from \eqref{eq: B-S 2}, with $E$ and $F$ swapped, that
\begin{equation} \label{eq: B-S 3}
    |\inn{U_{\gamma}\chi_E}{\chi_F}| = |\inn{\chi_E}{U_{\gamma}^*\chi_F}| \lesssim |E|^{1/\sigma_n'} |F|^{1/\rho_n}.
\end{equation}
A simple computation shows that the exponents $\rho_n$ and $\sigma_n$ satisfy
\begin{equation*}
    \begin{pmatrix}
        \frac{1}{q_n'} \\[4pt]
        \frac{1}{q_n}
    \end{pmatrix}
    = \frac{1}{2}
    \begin{pmatrix}
        \frac{1}{\rho_n} \\[4pt]
        \frac{1}{\sigma_n}
    \end{pmatrix} +
    \frac{1}{2}
    \begin{pmatrix}
        \frac{1}{\sigma_n'} \\[4pt]
        \frac{1}{\rho_n'}
    \end{pmatrix}.
\end{equation*}
Consequently, the desired inequality \eqref{eq: B-S 1} follows by the Marcinkiewicz interpolation theorem, using the restricted weak-type estimates \eqref{eq: B-S 2} and \eqref{eq: B-S 3}.

Turning to the proof of \eqref{eq: B-S 2}, first consider the case $|E|^{1/n(n+1)}|F|^{1/2} \leq 1$. Here \eqref{eq: B-S 2} holds due to trivial considerations. Indeed, as in the proof of Proposition~\ref{prop: weak endpoint}, we have $|\inn{U_{\gamma}\chi_E}{\chi_F}| \lesssim_B |E||F|$, which allows us to bound
\begin{equation*}
|\inn{U_{\gamma}\chi_E}{\chi_F}| \lesssim_B |E|^{1/\rho_n}|E|^{1/\rho_n'}|F| \leq |E|^{1/\rho_n} |F|^{- \frac{n(n+1)}{2\rho_n'}} |F| = |E|^{1/\rho_n}|F|^{1/\sigma_n'}, 
\end{equation*}
as required. 

Henceforth, assume $|E|^{1/n(n+1)}|F|^{1/2} \geq 1$. For $\lambda \geq 1$, by the Cauchy--Schwarz inequality,
\begin{equation}\label{eq: B-S 4}
    |\inn{U_{\low}^{\lambda}\chi_E}{\chi_F}| \leq \|U_{\low}^{\lambda}\chi_E\|_{L^2(\R^n)}|F|^{1/2}.
\end{equation}
It follows from Proposition~\ref{prop: low} that the multiplier $m_{\low}^{\lambda}$ associated to $U_{\low}^{\lambda}$ satisfies 
\begin{equation*}
   \|m_{\low}^{\lambda}\|_{L^{\infty}(\widehat{\R}^n)} = \|U_{\low}^{\lambda}\|_{L^2(\R^n) \to L^2(\R^n)}\lesssim_B \lambda^{\frac{n(n+1)}{2} - 1}. 
\end{equation*}
 Thus, by Plancherel's theorem,
\begin{equation} \label{eq: B-S 5}
    \|U^\lambda_{\low}\chi_E\|_{L^2(\R^n)}^2 \lesssim_B \lambda^{\frac{n(n+1)}{2} - 1} \int_{\widehat{\R}^n} |m^{\lambda}_{\low}(\xi)| |\widehat{\chi}_E(\xi)|^2\,\ud \xi.
\end{equation}

For $t \in [-2,2]$, let $M_t$ be the matrix with columns $\gamma^{(1)}(t), \dots, \gamma^{(n)}(t)$, and $D_\lambda$ the diagonal matrix satisfying $D_{\lambda}e_j = \lambda^j e_j$ for $1 \leq j \leq n$. For $\xi \in \widehat{\R}^n$, define the perturbed curves 
\begin{equation*}
    \gamma^\lambda(\,\cdot\,;\xi) \colon [-2,2] \to \R^n, \qquad \gamma^\lambda(t;\xi) \coloneq \gamma(t) - M_t D_\lambda^{-1} \xi \qquad \textrm{for $t \in [-2,2]$.}
\end{equation*}
Note that there exists some $c > 0$, depending only on $n$ and $\gamma$, such that for $|\xi| \leq c \lambda$ the curve $\gamma^\lambda(\,\cdot\,;\xi)$ lies in $\fkG(2B)$. We shall show in Lemma~\ref{lem: foliate} below that
\begin{equation} \label{eq: B-S 6}
    \int_{\widehat{\R}^n} |m^{\lambda}_{\low}(\xi)| |\widehat{\chi}_E(\xi)|^2\,\ud \xi \lesssim_{B, N} \int_{\R^n} \int_{\R} | \widehat{\chi}_E( \gamma^\lambda(t; \xi) ) |^2 |\alpha(t)|^2\; \ud t\, (1 + |\xi|)^{-N}\, \ud \xi
\end{equation}
for all $N \in \N_0$. 

We decompose the right-hand integral in \eqref{eq: B-S 6} into two regions: $|\xi| \leq c \lambda$ and $|\xi| > c \lambda$. For a fixed $t \in [-2,2]$, by the Riemann--Lebesgue Lemma,
\begin{equation} \label{eq:1o2ij312oije21oije21oilabel}
    \int_{|\xi| > c \lambda} | \widehat{\chi}_E( \gamma^\lambda(t; \xi) ) |^2 (1 + |\xi|)^{-N}\, \ud \xi \lesssim \lambda^{-N} \| \widehat{\chi}_E \|_{L^\infty(\R^n)}^2 \leq \lambda^{-N} |E|^2.
\end{equation}
Plancherel's theorem, after a change of variables $\xi \mapsto M_t^{-1} D_\lambda \xi$, guarantees that
\begin{align}
    \int_{|\xi| > c \lambda} | \widehat{\chi}_E( \gamma^\lambda(t; \xi) ) |^2 (1 + |\xi|)^{-N}\, \ud \xi &\lesssim_B \lambda^{\frac{n(n+1)}{2} -N} \int_{\widehat{\R}^n} |\widehat{\chi}_E(\xi)|^2\, \ud \xi \nonumber\\
    &\lesssim \lambda^{\frac{n(n+1)}{2} -N} |E| \label{eq:123pk4j1odije2of233421341412label}
\end{align}
Thus combining \eqref{eq:1o2ij312oije21oije21oilabel} and \eqref{eq:123pk4j1odije2of233421341412label}, for the global contribution, provided $N = O(1)$ is sufficiently large, 
\begin{align}
  \int_{|\xi| > c \lambda}  \int_{\R} & | \widehat{\chi}_E( \gamma^\lambda(t; \xi) ) |^2 |\alpha(t)|^2\; \ud t\, (1 + |\xi|)^{-N}\, \ud \xi \nonumber\\
  &\quad\quad\quad\quad\quad \lesssim_B \lambda^{-\frac{n(n+1)}{2} - 1} \min\{|E|, 1\}|E| \label{eq: B-S 7}.  
\end{align}

For the local contribution, fixing $|\xi| \leq c \lambda$, it follows by duality that
\begin{equation*}
  \int_{\R} |\widehat{\chi}_E(\gamma^\lambda(t;\xi))|^2 |\alpha(t)|^2\; \ud t  = \int_{\R^n} \overline{\chi_E(x)} (U_{\gamma^\lambda(\;\cdot\;;\xi)} \chi_E)(x)\; \ud x.  
\end{equation*}
Thus, by the Lorentz-space H\"{o}lder inequality and Proposition~\ref{prop: weak endpoint} applied to the perturbed curve $\gamma^\lambda(\; \cdot\; ; \xi)$, which lies in $\fkG(2B)$,
\begin{align}
    \int_{\R} |\widehat{\chi}_E(\gamma^\lambda(t;\xi))|^2|\alpha(t)|^2\; \ud t &\lesssim \| \chi_E \|_{L^{q_n',1}(\R^n)} \| U_{\gamma^\lambda(\;\cdot\;;\xi)} \chi_E \|_{L^{q_n,\infty}(\R^n)} \nonumber\\
    &\lesssim_B |E|^{2/q_n'}. \label{eq: B-S 8}
\end{align}

Combining \eqref{eq: B-S 4}, \eqref{eq: B-S 5} and \eqref{eq: B-S 6} with \eqref{eq: B-S 7} and \eqref{eq: B-S 8} yields
\begin{equation}\label{eq: B-S 9}
\begin{split}
|\inn{U_{\low}^{\lambda}\chi_E}{\chi_F}| &\lesssim_B \lambda^{-1}\min\{|E|^{1/2}, 1\}|E|^{1/2}|F|^{1/2}\\
&\quad\quad\quad\quad\quad\quad + \lambda^{\frac{1}{2} \left( \frac{n(n+1)}{2} - 1 \right)} |E|^{1/q_n'}|F|^{1/2}.
\end{split}
\end{equation}
Using the hypothesis of the case, $|F|^{1/2} \leq |E|^{1/n(n+1)}|F|$ and so
\begin{equation*}
   \min\{|E|^{1/2}, 1\}|E|^{1/2}|F|^{1/2} \leq  \min\{|E|^{1/2}, 1\}|E|^{1/2 + 1/n(n+1)}|F| \leq |E||F|. 
\end{equation*}
Thus, \eqref{eq: B-S 9} implies
\begin{equation}\label{eq: B-S 10}
 |\inn{U_{\low}^{\lambda}\chi_E}{\chi_F}| \lesssim_B \lambda^{-1}|E||F| +   \lambda^{\frac{1}{2} \left( \frac{n(n+1)}{2} - 1 \right)} |E|^{1/q_n'}|F|^{1/2}.   
\end{equation}
On the other hand, we know from Proposition~\ref{prop: high} that
\begin{equation}\label{eq: B-S 11}
 |\inn{U_{\high}^{\lambda}\chi_E}{\chi_F}| \lesssim_B \lambda^{-1} |E||F|.   
\end{equation}

We are free to choose $\lambda \geq 1$ in both \eqref{eq: B-S 10} and \eqref{eq: B-S 11}. Since, by hypothesis, $|E|^{\frac{1}{n(n+1)}}|F|^{1/2} \geq 1$, we may take
\begin{equation*}
    \lambda = \left[ |E|^{\frac{1}{n(n+1)}}|F|^{1/2} \right]^{\frac{4}{n^2 + n + 2}}
\end{equation*}
so that the right-hand sides of \eqref{eq: B-S 9} and \eqref{eq: B-S 11} agree. From this, we immediately deduce \eqref{eq: B-S 3}, which concludes the proof. 
\end{proof}

\begin{lemma} \label{lem: foliate}
    For the multiplier $m^\lambda_\low$ defined above and $E \subseteq \R^n$ measurable with finite measure, we have
    \begin{equation*}
\int_{\widehat{\R}^n} |m^{\lambda}_{\low}(\xi)| |\widehat{\chi}_E(\xi)|^2\,\ud \xi \lesssim \int_{\R^n} \int_{\R}  | \widehat{\chi}_E( \gamma^\lambda(t; \xi) ) |^2 |\alpha(t)|^2\, \ud t\, (1 + |\xi|)^{-2n}\, \ud \xi.
    \end{equation*}
\end{lemma}
\begin{proof}
    For $\mu \in \fkM(\lambda)$, recall that $K^\lambda_\mu = K[a_\mu]$ and \eqref{eq: low 1}. By a change of variables,
\begin{equation}\label{eq: foliate 1}
\widehat{K}^\lambda_\mu(\xi)= \frac{\lambda^{\frac{n(n+1)}{2}}}{|\det M_{t_\mu} |} \int_{\R} \zeta_\mu(t) v^\lambda_\mu(\xi;t) |\alpha(t)|^2\, \ud t
\end{equation}
where
\begin{equation*}
v^\lambda_\mu(\xi;t) \coloneq u\left(D_\lambda M_{t_\mu}^{-1} (\xi - \gamma(t))\right) \quad \textrm{for} \quad   u(\xi) \coloneq C^n  \prod_{r = 1}^n \widecheck{\eta}_\circ \left( C \xi_r \right).  
\end{equation*} 
Using the rapid decay of the $\widecheck{\eta}_\circ$, we have
\begin{equation} \label{eq: foliate 2}
    | u(\xi) | \lesssim_{\gamma, N} (1 + |\xi|)^{-N} \quad\text{for all $N \in \N_0$.}
\end{equation}

For $E \subseteq \R^n$ measurable with finite measure, it follows from \eqref{eq: foliate 1} that
\begin{multline*}
   \int_{\widehat{\R}^n} |\widehat{\chi}_E(\xi)|^2 |\widehat{K}^\lambda_\low(\xi)|\; \ud \xi \\ \leq  \int_{\widehat{\R}^n} |\widehat{\chi}_E(\xi)|^2 \int_{\R}  \sum_{\mu \in \fkM(\lambda)} \frac{\lambda^{\frac{n(n+1)}{2}}}{|\det(M_{t_\mu})|} \left|\zeta_\mu(t) v^\lambda_\mu(\xi;t) \right||\alpha(t)|^2\, \ud t\, \ud \xi.
\end{multline*}
Recall that $\gamma^\lambda(t;\xi) \coloneq \gamma(t) - M_t D_\lambda^{-1} \xi$. By changing the order of integration and applying a change of variables, we may express the right-hand side of the above display as 
\begin{equation*}
    \int_{\widehat{\R}^n} \int_{\R} |\widehat{\chi}_E(\gamma^\lambda(t;\xi))|^2 \sum_{\mu \in \fkM(\lambda)} \frac{|\det M_t|}{|\det M_{t_\mu}|} \left| \zeta_\mu(t) \,u\left( D_\lambda M_{t_\mu}^{-1} M_t D_\lambda^{-1} \xi\right)  \right| |\alpha(t)|^2 \ud t\, \ud \xi.
\end{equation*}
Note that the Jacobian factor
satisfies $|\det M_t| |\det M_{t_\mu}|^{-1} \lesssim_{\gamma} 1$ for all $t \in [-2,2]$ and $\mu \in \fkM(\lambda)$. In particular, 
\begin{multline}\label{eq: foliate 3}
    \int_{\widehat{\R}^n} |\widehat{\chi}_E(\xi)|^2 |\widehat{K}^\lambda_\low(\xi)|\; \ud \xi \\ \lesssim_B \int_{\widehat{\R}^n} \int_{\R} |\widehat{\chi}_E(\gamma^\lambda(t;\xi))|^2 \sum_{\mu \in \fkM(\lambda)} \left| \zeta_\mu(t) \,u\left( D_\lambda M_{t_\mu}^{-1} M_t D_\lambda^{-1} \xi\right) \right||\alpha(t)|^2 \,\ud t\, \ud \xi  .
\end{multline}

We perform a Taylor expansion to estimate $D_\lambda M_{t_\mu}^{-1} M_t D_\lambda^{-1}$. Let $M_{t_{\mu}}^{(k)}$ be the matrix obtained by taking the $k$th derivative of $M_t$ entrywise and evaluating at $t = t_{\mu}$. Then $M_{t_{\mu}}^{(k)}$ has columns $\gamma^{(1+k)}(t_{\mu}), \dots, \gamma^{(n+k)}(t_{\mu})$. So for $1 \leq k \leq n-1$, we deduce that
\begin{equation*}
    M_{t_{\mu}}^{-1} M_{t_{\mu}}^{(k)} = \begin{pmatrix} 0 & * \\ I_{n-k} & * \end{pmatrix}. 
\end{equation*}
For $1 \leq a$, $b \leq n$, we conclude that $(D_\lambda M_{t_\mu}^{-1} M_{t_\mu}^{(k)} D_\lambda^{-1})_{ab} = \lambda^{a - b} (M_{t_\mu}^{-1} M_{t_\mu}^{(k)})_{ab}$ is equal to zero unless $1 \leq b \leq n - k$ and $a - b = k$ or $n - k <b  \leq n$. Thus, each entry of the matrix is $O(\lambda^k)$. On the other hand, for each $s \in [-2,2]$, every entry of the matrix $D_\lambda M_{t_{\mu}}^{-1} M_s^{(n)} D_\lambda^{-1}$ is $O_B(\lambda^{n-1})$. Taylor expansion allows us to write $D_\lambda M_{t_\mu}^{-1} M_t D_\lambda^{-1}$ as
\begin{equation*}
  I + \sum_{k = 1}^{n-1} D_\lambda M_{t_\mu}^{-1} M_{t_\mu}^{(k)} D_\lambda^{-1} \frac{(t - t_\mu)^k}{k!} + O \left( |t - t_\mu|^n \sup_{s \in [-2,2]} D_\lambda M_{t_{\mu}}^{-1} M_s^{(n)} D_\lambda^{-1} \right).   
\end{equation*} 
Suppose $t \in \supp{}\zeta_\mu$, so that $|t - t_\mu| \leq (C\lambda)^{-1}$. From this, we deduce that
\begin{equation}\label{eq: foliate 4}
    |D_\lambda M_{t_\mu}^{-1} M_t D_\lambda^{-1} - I| \lesssim_{\gamma} C^{-1} \quad \textrm{and therefore} \quad | (D_\lambda M_{t_\mu}^{-1} M_t D_\lambda^{-1}) \xi | \gtrsim |\xi|,
\end{equation}
provided $C \geq 1$ is chosen sufficiently large, depending only on $n$ and $B$.

The sets $\supp{}\zeta_\mu$ are boundedly overlapping over $\mu \in \fkM(\lambda)$ and $|\zeta_\mu(t)| \leq 1$ for all $t \in [-1,1]$. Combining these observations with \eqref{eq: foliate 2} and \eqref{eq: foliate 4}, it follows that
\begin{equation*}
    \sum_{\mu \in \fkM(\lambda)} \left| \zeta_\mu(t)\, u\left( D_\lambda M_{t_\mu}^{-1} M_t D_\lambda^{-1} \xi\right)  \right| \lesssim_{B, N} (1 + |\xi|)^{-N}. 
\end{equation*} 
Substituting this bound into \eqref{eq: foliate 3} completes the proof.
\end{proof}

%%%%%%%%%%%%%%%%%%%%%%%%%%%%%%%%%%%%%%%%%%%%%%%%%%%%%%%%%%%%%%%%%%%%%%%%%%%%%%%%%%%%%%%%%%%%%%%%

%    Square function bounds related to orthonormal Fourier restriction

%%%%%%%%%%%%%%%%%%%%%%%%%%%%%%%%%%%%%%%%%%%%%%%%%%%%%%%%%%%%%%%%%%%%%%%%%%%%%%%%%%%%%%%%%%%%%%%%

\section{Square functions related to orthonormal Fourier restriction}

The goal of this section is to prove Theorem~\ref{thm: square fn}. 

%%%%%%%%%%%%%%%%%%%%%%%%%%%%%%%%%%%%%%%%%%%%%%%%%%%%%%%%%%%%%%%%%%%%%%%%%%%%%%%%%%%%%%%%%%%%%%%%

%   Wave packets

%%%%%%%%%%%%%%%%%%%%%%%%%%%%%%%%%%%%%%%%%%%%%%%%%%%%%%%%%%%%%%%%%%%%%%%%%%%%%%%%%%%%%%%%%%%%%%%%

\subsection{Wave packets} We begin by giving the precise definition of a \textit{wave packet}, and the notion of \textit{direction-separated plates}, as featured in Theorem~\ref{thm: square fn}. Fix $\gamma \colon [-2,2] \to \R^n$ a smooth, nondegenerate curve.

\begin{definition} For $R \geq 1$ and $0 < \delta < 1$, an \textit{$(R, \delta)$-plate associated to $\gamma$} is a set of the form
\begin{equation*}
    T = \Big\{ x \in \R^n : \big|(x-x_T) \cdot \gamma^{(r)}(s_T)\big| < R\, \delta^{\frac{n-r}{n-1}} \textrm{ for $1 \leq r \leq n$} \Big\}
\end{equation*}
for some $x_T \in \R^n$ and $s_T \in [-1,1]$. We refer to $x_T$ as the \textit{centre} $T$ and $s_T$ as the \textit{direction} of $T$. If $R = 1$, then we simply refer to $T$ as $\delta$-plate associated to $\gamma$. 
\end{definition}

We now introduce the notion of a \textit{wave packet}, as featured in the statement of Theorem~\ref{thm: square fn}. For this purpose, fix a smooth, non-negative bump function $\zeta$ satisfying $\supp{} \zeta \subseteq [-1/2, 1/2]$ and normalised so that $\int_{\R} \zeta(t)^2\; \ud t = 1$. 

\begin{definition}\label{dfn: wp} Let $R \geq 1$ and $\delta \coloneq R^{- \frac{n-1}{n}}$ and let $T$ be a $(R, \delta)$-plate associated to $\gamma$. We define the \textit{$L^2$-normalised wavepacket associated to $T$} by the formula
\begin{equation}\label{eq: wp}
    g_T(s) \coloneq R^{1/2n}e^{-2\pi i x_T \cdot \gamma(s)} \zeta(R^{1/n}(s - s_T)) \qquad \textrm{for all $T \in \T$.}
\end{equation} 
\end{definition}

The system $(g_T)_{T \in \bdT}$ in Theorem~\ref{thm: square fn} is understood to be comprised of functions of the above form, associated to some fixed family of $(R, R^{- \frac{n-1}{n}})$-plates $\bdT$ for some fixed value $R \geq 1$. The plates in $\bdT$ are also assumed to be \textit{direction-separated}, in the following sense. 

\begin{definition}\label{dfn: dir sep} For $R \geq 1$ and $0 < \delta < 1$, let $T_1$, $T_2$ be $(R,\delta)$-plates associated to $\gamma$. We say $T_1$, $T_2$ are \textit{direction-separated} if $|s_{T_1} - s_{T_2}| > \delta^{\frac{1}{n-1}}$. If $\bdT$ is a set of $(R, \delta)$-plates, then we say $\bdT$ is  \textit{direction-separated} if $T_1$, $T_2$ are direction separated for any distinct $T_1$, $T_2 \in \bdT$.
\end{definition}

Suppose $\bdT$ is a direction-separated family of $(R, R^{- \frac{n-1}{n}})$-plates. By definition, the family of wave packets $(g_T)_{T\in \bdT}$ form an orthonormal family, since their supports are disjoint. The freedom to choose the oscillation parameter $x_T$ in \eqref{eq: wp} ensures that families of direction seperated wave packets form a rich class of testing functions. In this way, we can understand Theorem~\ref{thm: square fn} as partial progress towards Conjecture~\ref{conj: CDH}. 

%%%%%%%%%%%%%%%%%%%%%%%%%%%%%%%%%%%%%%%%%%%%%%%%%%%%%%%%%%%%%%%%%%%%%%%%%%%%%%%%%%%%%%%%%%%%%%%%

%   Reduction to a Kakeya-type estimate

%%%%%%%%%%%%%%%%%%%%%%%%%%%%%%%%%%%%%%%%%%%%%%%%%%%%%%%%%%%%%%%%%%%%%%%%%%%%%%%%%%%%%%%%%%%%%%%%

\subsection{Reduction to a Kakeya-type estimate} Theorem~\ref{thm: square fn} is a direct consequence of a geometric inequality for $\delta$-plates. The latter is valid under the following weak separation hypothesis.

\begin{definition} Let $0 < \delta < 1$ and $T_1$, $T_2$ be $\delta$-plates associated to $\gamma$. We say $T_1$, $T_2$  are \textit{essentially distinct} if either of the following hold:
\begin{enumerate}[i)]
\item $|s_{T_1} - s_{T_2}| > \delta^{\frac{1}{n-1}}$;
\item $|s_{T_1} - s_{T_2}| \leq \delta^{\frac{1}{n-1}}$ and
\begin{equation*}
    \max_{1 \leq r \leq n} \delta^{-\frac{n-r}{n-1}}|(x_{T_1} - x_{T_2}) \cdot \gamma^{(r)}(s_{T_i})| > 1 \qquad \textrm{for $i \in \{ 1, 2 \}$.}
\end{equation*}
\end{enumerate}
\end{definition}

Defining the exponents
\begin{equation*}
s_n \coloneq \frac{q_n}{2} = \frac{n(n+1)}{2} \quad \textrm{and} \quad    e(n; s) \coloneq \frac{n^2}{n-1}\cdot \frac{1}{s} - \frac{1}{n-1} \quad \textrm{for $s_n \leq s \leq \infty$,}
\end{equation*}
the key geometric inequality reads thus. 

\begin{proposition}\label{prop: Kakeya} For $s_n < s \leq \infty$, the inequality
\begin{equation*}
    \left\|\sum_{T \in \T} \nu_T\chi_T \right\|_{L^s(\R^n)} \lesssim_{\gamma, s} \delta^{e(n;s)} \left(\sum_{T \in \T} |\nu_T|^{s/n}\right)^{n/s}
\end{equation*}
holds whenever $\T$ is a family of essentially distinct $\delta$-plates associated to $\gamma$ and $(\nu_T)_{T \in \T}$ are complex coefficients. 
\end{proposition}

We shall deduce Theorem~\ref{thm: square fn} from  Proposition~\ref{prop: Kakeya}, using the Arkhipov--Chubarikov--Karatsuba estimate in Theorem~\ref{thm: ACK}. The argument is somewhat similar to that used to prove \eqref{eq: ACK nec} in \cite[p.21, Theorem 1.3]{ACK_book} (see also \cite{dAW}). 

\begin{proof}[Proof (of Theorem~\ref{thm: square fn}, assuming Proposition~\ref{prop: Kakeya})] It suffices to prove the theorem for $q > q_n$ and $r = q/2n$, since the remaining cases follow by the nesting of $\ell^r$.

For $R \geq 1$ and $\delta \coloneq R^{- \frac{n-1}{n}}$, let $\bdT$ be a direction-separated family of $(R, \delta)$-plates  and $(\nu_T)_{T \in \bdT}$ a sequence of complex coefficients. For each $T \in \bdT$, let $J(T) \coloneq \{s \in \R : |s - s_T| \leq (2R^{1/n})^{-1}\}$. Then $J(T)$ is an interval containing the support of $g_T$. It follows from the Arkhipov--Chubarikov--Karatsuba estimate in Theorem~\ref{thm: ACK} that 
\begin{align}\nonumber
    |\cE_{\gamma} g_T(x)| &= R^{1/2n} \left| \int_{\R} e^{2 \pi i (x-x_T) \cdot \gamma(s)} \zeta(R^{1/n} (s - s_T))\; \ud s \right|\\
    \label{eq: Kak vs rest 1}
    &\lesssim R^{1/2n} \big(R^{1/n} + H_{\gamma, J(T)}(x - x_T)\big)^{-1}. 
\end{align}

To estimate the right-hand side of \eqref{eq: Kak vs rest 1}, we define the level sets
\begin{align*}
E_0[T] &\coloneq \big\{ x \in \R^n : H_{\gamma, J(T)}(x - x_T) \leq R^{1/n} \big\}, \\
E_{\ell}[T] &\coloneq \big\{ x \in \R^n : (2^{\ell-1}R)^{1/n} < H_{\gamma, J(T)}(x - x_T) \leq (2^{\ell}R)^{1/n} \big\}, \qquad \ell \in \N.
\end{align*}
For each $\ell \in \N_0$, let $R_{\ell} \coloneq 2^n2^{\ell}R$ and $\delta_{\ell} \coloneq R_{\ell}^{- \frac{n-1}{n}}$. For $0 < c_{\gamma} < 1$ a constant chosen sufficiently small, depending only on $\gamma$ and $n$, to satisfy the forthcoming requirements of the proof, let $\cJ_{\ell}(T)$ be a maximal $c_{\gamma}\delta_{\ell}^{1/(n-1)}$-separated subset of $J(T)$. We define $\bdT_{\ell}[T]$ to be the set of all $(R_{\ell}, \delta_{\ell})$-plates $T_{\ell}$ with common centre $x_{T_{\ell}} \coloneq x_T$ and directions $s_{T_{\ell}}$ varying over $\cJ_{\ell}(T)$. Note that $\#\bdT_{\ell}[T] = \# \cJ_{\ell}(T) \lesssim_{\gamma} 2^{\ell/n}$. We further define 
\begin{equation*}
\bdT_{\ell} \coloneq \bigcup_{T \in \bdT} \bdT_{\ell}[T].
\end{equation*}
The above definitions satisfy the following properties:
\begin{enumerate}
    \item For all $T \in \bdT$ and $\ell \in \N_0$, we have 
    \begin{equation}\label{eq: Kak vs rest 2}
    E_{\ell}[T] \subseteq \bigcup_{T_{\ell} \in \bdT_{\ell}[T]} T_{\ell};
\end{equation}
    \item For all $\ell \in \N_0$, we can write $\bdT_{\ell}$ as a union of $O_{\gamma}(1)$ direction-separated subfamilies. 
\end{enumerate}

To see Property (1), let $x \in E_{\ell}[T]$ so there exists some $s_x \in J(T)$ such that $|(x - x_T)\cdot \gamma^{(r)}(s_x)| \leq 2^{-r}R_{\ell}^{r/n}$ for $1 \leq r \leq n$. By the definition of $\cJ_{\ell}(T)$, there exists some $s \in \cJ_{\ell}(T)$ such that $|s_x - s| \leq c_{\gamma}R_{\ell}^{-1/n}$. By Taylor's theorem,
\begin{equation*}
  |(x-x_T) \cdot \gamma^{(r)}(s)| \leq \sum_{j = r}^{n-1} |(x-x_T) \cdot \gamma^{(j)}(s_x)||s-s_x|^{j-r} + O_{\gamma}\big(R_{\ell}|s - s_x|^{n-r}\big),  
\end{equation*}
where the error term is controlled by using the nondegeneracy condition to bound
\begin{equation*}
    |x-x_T| \sim_{\gamma} \sum_{j=1}^n |(x-x_T) \cdot \gamma^{(j)}(s_x)| \lesssim R_{\ell}.
\end{equation*}
It therefore follows that 
\begin{equation*}
|(x-x_T) \cdot \gamma^{(r)}(s)| \leq R_{\ell}^{r/n}(1/2 + O_{\gamma}(c_{\gamma})) \qquad \textrm{for $1 \leq r \leq n$.}    
\end{equation*}
Let $T_{\ell} \in \bdT_{\ell}[T]$ be the $(R_{\ell}, R_{\ell}^{-(n-1)/n})$-plate with centre $x_{T_{\ell}} \coloneq x_T$ and direction $s_{T_{\ell}} \coloneq s$. Provided $c_{\gamma} > 0$ is chosen sufficiently small, it follows from the above that $x \in T_{\ell}$. This establishes \eqref{eq: Kak vs rest 2}.

On the other hand, property (2) follows immediately from the direction-separation of $\bdT$ and the above definition. In particular, the direction separation of the tubes $\bdT$ implies that the intervals $\{ J(T) : T \in \bdT \}$ are almost-disjoint. In addition, for each $T \in \bdT$, the set $\bdT_{\ell}[T]$ contains one plate for each direction in $\cJ_\ell(T)$, and $\cJ_\ell(T)$ can be written  as a union of $O_\gamma(1)$ sets which are $\delta_{\ell}^{1/(n-1)}$ separated. 

For each $T \in \bdT$, square both sides of \eqref{eq: Kak vs rest 1}. If $x \in E_{\ell}[T]$ for some $\ell \in \N_0$, then the resulting right-hand side is bounded by $R^{1/n}(2^{\ell}R)^{-2/n} \sim_{\gamma} 2^{-\ell/n} R_{\ell}^{-1/n}$. Combining this observation with \eqref{eq: Kak vs rest 2}, we deduce that
\begin{equation}\label{eq: Kak vs rest 3}
    |\nu_T||\cE_{\gamma} g_T(x)|^2 \lesssim \sum_{\ell = 0}^{\infty} 2^{-\ell/n} R_{\ell}^{-1/n} \sum_{T_{\ell} \in \bdT_{\ell}[T]} |\nu_{T_{\ell}}| \, \chi_{T_{\ell}}(x),
\end{equation}
where we define $\nu_{T_{\ell}} \coloneq \nu_T$ for all $T_{\ell} \in \bdT_{\ell}[T]$. 

Let $\T_{\ell} \coloneq \{R_{\ell}^{-1} \cdot T : T \in \bdT_{\ell}\}$ be the set of $\delta_{\ell}$-plates formed by applying the dilation map $x \mapsto R_{\ell}^{-1} x$ to each $T \in \bdT_{\ell}$. We then define $\nu_{T_{\ell}} \coloneq \nu_{R_{\ell} \cdot T_{\ell}}$ for all $T_{\ell} \in \T_{\ell}$. Summing \eqref{eq: Kak vs rest 3} over $T \in \bdT$ and rescaling, we conclude that
\begin{equation*}
    \sum_{T \in \bdT} |\nu_T||\cE_{\gamma} g_T(x)|^2 \lesssim \sum_{\ell = 0}^{\infty} 2^{-\ell/n} R_{\ell}^{-1/n} \sum_{T_{\ell} \in \T_{\ell}} |\nu_{T_{\ell}}|\, \chi_{T_{\ell}}(R_{\ell}^{-1} x).
\end{equation*}
Consequently, with $s \coloneq q/2$, we have
\begin{equation}\label{eq: Kak vs rest 4}
        \left\| \sum_{T \in \bdT}\nu_T |\cE_{\gamma} g_T|^2 \right\|_{L^{q/2}(\R^n)} \lesssim \sum_{\ell = 0}^{\infty} 2^{-\ell/n} R_{\ell}^{-1/n} R_{\ell}^{n/s} \left\|\sum_{T \in \T_{\ell}}|\nu_{T_{\ell}}|\, \chi_{T_{\ell}} \right\|_{L^s(\R^n)}.
\end{equation}
By property $(2)$ above, $\T_{\ell}$ can be written as a union of $O_{\gamma}(1)$ direction-separated families of $\delta_{\ell}$-plates. Proposition~\ref{prop: Kakeya} implies that for each $\ell \in \N_0$, we have
\[ \left\|\sum_{T \in \T_{\ell}}|\nu_{T_{\ell}}|\, \chi_{T_{\ell}} \right\|_{L^s(\R^n)} \lesssim R_\ell^{ - (\frac{n-1}{n}) e(n;s)} \left( \sum_{T \in \T_\ell} |\nu_T|^{s/n} \right)^{n/s}. \]
We may use this together with the triangle inequality to bound the norms on the right-hand side of \eqref{eq: Kak vs rest 4}. A simple computation show that the resulting $R_{\ell}$ exponent satisfies
\begin{equation*}
    -\frac{1}{n} + \frac{n}{s} - \frac{n-1}{n}\cdot e(n;s) = 0.
\end{equation*}
On the other hand, 
\begin{equation*}
    \sum_{T_{\ell} \in \T_{\ell}} |\nu_{T_{\ell}}|^{s/n} = \sum_{T \in \bdT} \sum_{T_{\ell} \in \bdT_{\ell}[T]} |\nu_{T_{\ell}}|^{s/n} = \sum_{T \in \bdT} |\nu_T|^{s/n} \#\bdT_{\ell}[T] \lesssim_{\gamma} 2^{\ell/n}\sum_{T \in \bdT} |\nu_T|^{s/n} . 
\end{equation*}
Thus, altogether we get 
\begin{equation*}
    \left\|\sum_{T \in \bdT}\nu_T |\cE_{\gamma} g_T|^2 \right\|_{L^{q/2}(\R^n)} \lesssim_{\gamma} \sum_{\ell = 0}^{\infty} 2^{-\ell(1/n - 1/s)} \left( \sum_{T \in \bdT} |\nu_T|^{q/2n} \right)^{2n/q}.
\end{equation*}
Since $s > n$, the above geometric series is summable. 
\end{proof}

%%%%%%%%%%%%%%%%%%%%%%%%%%%%%%%%%%%%%%%%%%%%%%%%%%%%%%%%%%%%%%%%%%%%%%%%%%%%%%%%%%%%%%%%%%%%%%%%

%   Preliminaries for the proof of the Kakeya-type estimate

%%%%%%%%%%%%%%%%%%%%%%%%%%%%%%%%%%%%%%%%%%%%%%%%%%%%%%%%%%%%%%%%%%%%%%%%%%%%%%%%%%%%%%%%%%%%%%%%

\subsection{Preliminaries for the proof of the Kakeya-type estimate} 

It remains to prove Proposition~\ref{prop: Kakeya}. In this subsection, we present some preliminaries for the argument. In \S\ref{subsec: combining} below, we combine the various ingredients discussed here with a Bourgain--Guth~\cite{BG2011} iteration scheme to complete the proof. 

%%%%%%%%%%%%%%%%%%%%%%%%%%%%%%%%%%%%%%%%%%%%%%%%%%%%%%%%%%%%%%%%%%%%%%%%%%%%%%%%%%%%%%%%%%%%%%%%

%   Quantified nondegeneracy 

%%%%%%%%%%%%%%%%%%%%%%%%%%%%%%%%%%%%%%%%%%%%%%%%%%%%%%%%%%%%%%%%%%%%%%%%%%%%%%%%%%%%%%%%%%%%%%%%

\subsubsection*{Quantified nondegeneracy} For technical reasons, it is convenient to work with a weakly quantified version of the nondegeneracy condition. Recall the class of smooth, non-degenerate curves $\fkG(B)$ for $B \geq 1$ from \S\ref{sub:preliminaries}. In what follows, we prove estimates which are uniform over all curves $\gamma \in \fkG(B)$ for some fixed $B \geq 1$.

%%%%%%%%%%%%%%%%%%%%%%%%%%%%%%%%%%%%%%%%%%%%%%%%%%%%%%%%%%%%%%%%%%%%%%%%%%%%%%%%%%%%%%%%%%%%%%%%

%   Multilinear estimates

%%%%%%%%%%%%%%%%%%%%%%%%%%%%%%%%%%%%%%%%%%%%%%%%%%%%%%%%%%%%%%%%%%%%%%%%%%%%%%%%%%%%%%%%%%%%%%%%

\subsubsection*{Multilinear estimates} We begin by proving an elementary $n$-linear variant of Proposition~\ref{prop: Kakeya}, under a transversality hypothesis. Let $0 < \delta \leq \nu \leq 1$ and $\T_1, \dots, \T_n$ be families of $\delta$-plates associated to $\gamma$. We say that $\T_1, \dots, \T_n$ are $\nu$-\textit{transversal} if
\begin{equation}\label{eq: trans}
    \big|\det\begin{pmatrix}
    \gamma'(s_{T_1}) & \cdots & \gamma'(s_{T_n}) 
    \end{pmatrix}
    \big| \geq  \nu \quad \text{for all $T_\ell \in \T_\ell$, $1 \leq \ell \leq n$.}
\end{equation}
With this definition, our multilinear estimate reads thus. 

\begin{proposition}\label{prop: n linear} For $0 < \delta \leq \nu \leq 1$ and $n \leq s \leq \infty$, the inequality
\begin{equation}\label{eq: n linear}
    \left\| \prod_{\ell=1}^n \bigg|\sum_{T_{\ell} \in \T_{\ell}} \chi_{T_{\ell}} \bigg|^{1/n} \right\|_{L^s(\R^n)} \lesssim \nu^{-1/s} \delta^{e(n;s)}[\#\T]^{n/s}
\end{equation}
    holds whenever $\T$ is a family of essentially distinct $\delta$-plates associated to $\gamma$ and $\T_1, \dots, \T_n \subseteq \T$ are $\nu$-transversal.
\end{proposition}

We remark that, using the nondegeneracy condition on $\gamma$, we can express the determinant in \eqref{eq: trans} in a convenient alternative form. In particular, suppose $\gamma \in \fkG(B)$
for some $B \geq 2$. It follows from \cite[Proposition 4.1]{GGPRY2021} that there exists some $0 < \rho_B \leq 2$, depending only on $B$, such that
\begin{equation}\label{eq: non-deg trans}
   \big|\det\big(
    \gamma'(s_1) \, \cdots \, \gamma'(s_n) 
    \big)
    \big| \sim \prod_{1 \leq i < j \leq n} |s_i - s_j| \quad \textrm{for $\displaystyle \max_{1 \leq i < j \leq n} |s_i - s_j| < \rho_B$.} 
\end{equation}
Hence, Proposition~\ref{prop: n linear} applies in situations where the directions between typical $n$-tuples of intersecting plates are well-separated. 

To prove Proposition~\ref{prop: n linear}, we make use of the following elementary observation. 

\begin{lemma}\label{lem: bdd mult} Let $0 < \delta < 1$ and $\T$ be a family of essentially distinct $\delta$-plates associated to $\gamma$. Then 
\begin{equation*}
    \bigg\| \sum_{T \in \T} \chi_T \bigg\|_{L^{\infty}(\R^n)} \lesssim \delta^{-\frac{1}{n-1}}. 
\end{equation*}
\end{lemma}

\begin{proof} Fix $x \in \R^n$ and let $\cI$ be a boundedly overlapping cover of $[-1,1]$ by closed intervals of length $\delta^{1/(n-1)}$, so that $\#\cI \lesssim \delta^{-1/(n-1)}$. It then follows that
\begin{equation*}
   \sum_{T \in \T} \chi_T(x) \lesssim  \delta^{- \frac{1}{n-1}}\, \max_{I \in \cI} \# \T[x; I] \quad \textrm{where} \quad \T[x; I] \coloneq \{T \in \T : x \in T,\, s_T \in I\}. 
\end{equation*}
It remains to show $\#\T[x; I] \lesssim 1$ for all $I \in \cI$. Fix $I \in \cI$ with $\T[x; I] \neq \varnothing$ and define a metric $d$ on $\R^n$ by
\begin{equation*}
    d(x,y) \coloneq \max_{s \in I} \max_{1 \leq r \leq n} \delta^{- \frac{n-r}{n-1}}|(x - y) \cdot \gamma^{(r)}(s)|, \qquad x,\,y \in \R^n.
\end{equation*}
Note that, because the metric balls $B_d(x,\rho)$ defined with respect to the metric $d$ satisfy $B_d(x,\rho) = \rho \cdot B_d(x,1)$, the metric is doubling, in the sense that $|B_d(x, 2\rho)| \lesssim_{\gamma} |B_d(x, \rho)|$ for all $x \in \R^n$ and $\rho > 0$. Suppose $T_1$, $T_2 \in \T[x; I]$ and $T_1 \neq T_2$. Then $|s_{T_1} - s_{T_2}| \leq \delta^{1/(n-1)}$. Since the $\delta$-plates in $\T$ are essentially distinct, it follows that $d(x_{T_1}, x_{T_2}) \geq 1$. In other words, the set of centres
\begin{equation*}
    \cC[x;I] \coloneq \big\{ x_T \in \R^n : T \in \T[x; I] \big\}
\end{equation*}
is $1$-separated and $T \mapsto x_T$ is a bijection from $\T[x;I]$ to $\cC[x;I]$. 

On the other hand, writing $x_i \coloneq x_{T_i}$ and $s_i \coloneq s_{T_i}$ for $i \in \{1,2\}$, for $1 \leq r \leq n$ we have
\begin{equation*}
    |(x_1 - x_2) \cdot \gamma^{(r)}(s)| \leq \sum_{i = 1}^2 |(x_i - x) \cdot \gamma^{(r)}(s_i)| + \big|(x_i - x) \cdot \big(\gamma^{(r)}(s) - \gamma^{(r)}(s_i)\big)\big|. 
\end{equation*}
Since $x \in T_i$, we have $|(x_i - x) \cdot \gamma^{(r)}(s_i)| < \delta^{\frac{n-r}{n-1}}$. On the other hand, by Taylor expansion, 
\begin{align*}
    \big|(x_i - x) \cdot \big(\gamma^{(r)}(s) - \gamma^{(r)}(s_i)\big)\big| &\lesssim_{B} \sum_{j = r+1}^{n-1} |(x_i - x)\cdot\gamma^{(j)}(s_i)| |s-s_i|^{j-r} + |s-s_i|^{n-r}\\
    &\lesssim \delta^{\frac{n-r}{n-1}},
\end{align*}
where we again use $x \in T_i$ and also $|s - s_i| \leq \delta^{1/(n-1)}$. Thus, we see that the diameter of $\cC[x; I]$ with respect to the $d$ metric is $O_{B}(1)$. The family of balls $\{ B_d(x,1/2): x \in \cC[x;I] \}$ are thus disjoint, yet contained in $B_d(x,O_{B}(1))$ for each $x \in \cC[x,I]$; the doubling property of the metric implies the size of the balls $B_d(\,\cdot\,,1/2)$ and $B_d(\,\cdot\,,O(1))$ are comparable, hence $\#\cT[x;I] =\#\cC[x;I] \lesssim_{\gamma} 1$, as required. 
\end{proof}

We now turn to the proof of the multilinear estimate. 

\begin{proof}[Proof (of Proposition~\ref{prop: n linear})] For $n \leq s \leq \infty$ and for $\nu$-transversal $\T_1, \dots, \T_n$ subsets of a family $\T$ of essentially distinct $\delta$-plates, we shall show
\begin{equation}\label{eq: n linear 1}
    \left\| \prod_{\ell=1}^n \bigg|\sum_{T_{\ell} \in \T_{\ell}} \chi_{T_{\ell}} \bigg|^{1/n} \right\|_{L^n(\R^n)} \lesssim \nu^{-1/n} \delta\, \#\T
\end{equation}
and
\begin{equation}\label{eq: n linear 2}
    \left\| \prod_{\ell=1}^n \bigg|\sum_{T_{\ell} \in \T_{\ell}} \chi_{T_{\ell}}\bigg|^{1/n}\right\|_{L^{\infty}(\R^n)} \lesssim \delta^{- \frac{1}{n-1}}, 
\end{equation}
corresponding to the $s = n$ and $s = \infty$ cases of \eqref{eq: n linear}, respectively. Once this is established, the desired conclusion follows by interpolation via logarithmic convexity property of $L^p$. Furthermore, the inequality \eqref{eq: n linear 2} is a trivial consequence of Lemma~\ref{lem: bdd mult}. Thus, the problem is further reduced to proving \eqref{eq: n linear 1} only. 

Given $T \in \T$, we let $S_T$ denote the slab
\begin{equation*}
    S_T \coloneq  \big\{ x \in \R^n : \big|(x-x_T) \cdot \gamma^{(1)}(s_T)\big| < \delta\big\}
\end{equation*}
formed by relaxing all but the first inequality in the definition of $T$. The transversality hypothesis guarantees that
    \begin{equation*}
        \left| \bigcap_{\ell = 1}^n T_{\ell} \right| \leq \left| \bigcap_{\ell = 1}^n S_{T_{\ell}} \right| \sim \nu^{-1} \delta^n \qquad \textrm{for $T_\ell \in \T_\ell$,  $1 \leq \ell \leq n$.}
    \end{equation*}
   Thus, by expanding out the $L^n(\R^n)$-norm, we conclude that
    \begin{equation*}
        \left\| \prod_{\ell=1}^n  \bigg| \sum_{T_\ell \in \T_{\ell}} \chi_{T_\ell} \bigg|^{1/n} \right\|_{L^n(\R^n)}^n = \sum_{T_1 \in \T_1} \cdots \sum_{T_n \in \T_n} \left| \bigcap_{\ell = 1}^n T_{\ell} \right| \lesssim \nu^{-1} \delta^n \prod_{\ell=1}^n \# \T_\ell.
    \end{equation*}
    Bounding $\#\T_{\ell} \leq \#\T$ for all $1 \leq \ell \leq n$ and taking the $n$th root then yields \eqref{eq: n linear 1}.
\end{proof}

%%%%%%%%%%%%%%%%%%%%%%%%%%%%%%%%%%%%%%%%%%%%%%%%%%%%%%%%%%%%%%%%%%%%%%%%%%%%%%%%%%%%%%%%%%%%%%%%

%   Angular vs spatial localisation 

%%%%%%%%%%%%%%%%%%%%%%%%%%%%%%%%%%%%%%%%%%%%%%%%%%%%%%%%%%%%%%%%%%%%%%%%%%%%%%%%%%%%%%%%%%%%%%%%

\subsubsection*{Angular vs spatial localisation } Let $B \geq 1$ and $\gamma \in \fkG(B)$. Further, let $0 < \delta \leq \eta \leq 1$ and $I \subseteq [-1,1]$ be an interval of radius $\eta^{1/(n-1)}$ centred at $s_I \in [-1,1]$. For any family $\T$ of $\delta$-plates associated to $\gamma$, define
\begin{equation*}
    \T[I] \coloneq \{ T \in \T : s_T \in I \}. 
\end{equation*}
We further define $\bbP[I]$ to be a family of $\eta$-plates, each with direction $s_I$, tiling $\R^n$ with disjoint interiors. We further decompose
\begin{equation*}
   \T[I] = \bigcup_{P \in \bbP[I]} \T[I; P] \qquad \textrm{where} \qquad  \T[I; P] \coloneq \{ T \in \T[I] : T \cap P \neq \emptyset\}.
\end{equation*}
Finally, given $C > 0$ and any $P \in \bbP[I]$, we let $C \cdot P$ denote the central dilate of $P$ by a factor of $C$; that is, for $s_P = s_I$, we have
\begin{equation*}
    C \cdot P \coloneq \big\{ x \in \R^n : |(x-x_P) \cdot \gamma^{(r)}(s_P)| \leq C \eta^{ \frac{n-r}{n-1}} \textrm{ for $1 \leq r \leq n$}\big\}.
\end{equation*}
With this definition, we have the following lemma, which tells us that the plates in $\T[I]$ naturally partition into non-interacting subfamilies. 

\begin{lemma}\label{lem: Kakeya orthog} With the above setup, there exists a constant $C_B \geq 1$ depending only on $B \geq 1$, such that if $T \in \T[I; P]$ , then $T \subseteq C_B \cdot P$. 
\end{lemma}

For the constant $C_B \geq 1$ guaranteed by Lemma~\ref{lem: Kakeya orthog}, we let $\overline{P} \coloneq C_B \cdot P$ whenever $P$ is an $\eta$-plate associated to $\gamma \in \fkG(B)$ for some $0 < \eta < 1$.

\begin{proof}[Proof (of Lemma~\ref{lem: Kakeya orthog})] Let $0 < \delta \leq \eta \leq 1$ and $I \subseteq [-1,1]$ be an interval of radius $\eta^{1/(n-1)}$. Fix $P \in \bbP[I]$, $T \in \T[I; P]$ and $z \in T$. The key observation is that
\begin{equation}\label{eq: Kakeya orthog 1}
    |(z - x_T) \cdot \gamma^{(r)}(s_P)| \lesssim_B \eta^{ \frac{n-r}{n-1}} \qquad \textrm{for $1 \leq r \leq n$.}
\end{equation}
The case $r = n$ of \eqref{eq: Kakeya orthog 1} is immediate. For the remaining cases,  note that the conditions $P \in \bbP[I]$ and $T \in \T[I;P]$ imply that 
\begin{equation*}
|s_T - s_P| \leq \eta^{\frac{1}{n-1}} \qquad \textrm{and} \qquad |(z - x_T) \cdot \gamma^{(r)}(s_T)| \leq \eta^{ \frac{n-r}{n-1} } \qquad \textrm{for $1 \leq r \leq n$.}   
\end{equation*}
It therefore follows by Taylor's theorem, using the upper bound on the $n$th derivative from \eqref{eq: non-deg 2}, that 
\begin{align*}
  |(z - x_T) \cdot \gamma^{(r)}(s_P)| &\leq \sum_{\ell = r}^{n-1} |(z - x_T) \cdot \gamma^{(r)}(s_T)|\frac{|s_P - s_T|^{r - \ell}}{(r - \ell)!} + O_B\big(|s_P - s_T|^{n-r}\big) \\
 & \lesssim_B \eta^{\frac{n-r}{n-1}}
\end{align*}
for $1 \leq r \leq n-1$, which establishes the remaining cases of \eqref{eq: Kakeya orthog 1}.

Since $T \in \T[I; P]$, we have $T \cap P \neq \emptyset$, and so we know there exists some $y \in T \cap P$. Given any $x \in T$, for $1 \leq r \leq n$ we obtain the bound
\begin{equation}\label{eq: Kakeya orthog 2}
\begin{split}
    |(x-x_P)\cdot \gamma^{(r)}(s_P)| &\leq |(x-x_T)\cdot \gamma^{(r)}(s_P)|\\
    &\quad\quad + |(y-x_T)\cdot \gamma^{(r)}(s_P)| + |(y-x_P)\cdot \gamma^{(r)}(s_P)|.
\end{split}
\end{equation}
By applying \eqref{eq: Kakeya orthog 1} with $z = x \in T$ and $z = y \in T$, we see that first two terms on the right-hand side of \eqref{eq: Kakeya orthog 2} are both $O_{\gamma}(\eta^{(n-r)/(n-1)})$. On the other hand, since $y \in P$, the third term on the right-hand side of \eqref{eq: Kakeya orthog 2} is $O(\eta^{(n-r)/(n-1)})$. Thus, 
\begin{equation*}
   |(x-x_P)\cdot \gamma^{(r)}(s_P)| \lesssim_B \eta^{\frac{n-r}{n-1}} \qquad \textrm{for all $1 \leq r \leq n$,} 
\end{equation*}
which means $x \in C_B \cdot P$ for a suitable choice of constant $C_B \geq 1$, depending only on $B$. 
\end{proof}

%%%%%%%%%%%%%%%%%%%%%%%%%%%%%%%%%%%%%%%%%%%%%%%%%%%%%%%%%%%%%%%%%%%%%%%%%%%%%%%%%%%%%%%%%%%%%%%%

%   Rescaling

%%%%%%%%%%%%%%%%%%%%%%%%%%%%%%%%%%%%%%%%%%%%%%%%%%%%%%%%%%%%%%%%%%%%%%%%%%%%%%%%%%%%%%%%%%%%%%%%

\subsubsection*{Rescaling} Continuing with the above setup, fix an $\eta$-plate $P \in \bbP[I]$ associated to $\gamma \in \fkG(B)$. The dilate $\overline{P}$ may be expressed as
\begin{equation*}
  \overline{P} =  \big\{ x \in \R^n : |\bdA_P(x)|_{\infty} \leq C_B \big\}
\end{equation*}
for a suitable choice of affine map $\bdA_P \colon \R^n \to \R^n$. More precisely, 
\begin{equation}\label{eq: rescaling 1}
    \bdA_P(x) \coloneq \big(\bdD_P\big)^{-1} \circ \big([\gamma]_P\big)^{\top} (x - x_P) \qquad \textrm{for all $x \in \R^n$,}
\end{equation}
where $\bdD_P$ is the $n \times n$ diagonal matrix which acts on the coordinate vectors so that $\bdD_P e_r = \eta^{(n-r)/(n-1)} e_r$ for $1 \leq r \leq n$, and 
\begin{equation*}
    [\gamma]_P \coloneq 
    \big(
        \gamma^{(1)}(s_P) \, \cdots \, \gamma^{(n)}(s_P)
    \big).
\end{equation*}
We further define the rescaled curve $\gamma_P \colon [-2,2] \to \R^n$ associated to $P$ by 
\begin{equation*}%\label{eq: rescaling 2}
    \gamma_P(s) \coloneq \eta^{- \frac{n}{n-1}} \bdD_P \circ \big([\gamma]_P\big)^{-1} \left(\gamma \big(s_P + \eta^{ \frac{1}{n-1} }s \big) - \gamma(s_P)\right) \quad s \in [-2,2].
\end{equation*}
This rescaled curve is also nondegenerate.

\begin{lemma}\label{lem: resc curve} With the above setup, there exists a constant $c_B  >0$, depending only on $B$, such that $\gamma_P \in \fkG(2)$ whenever $0 < \eta < c_B$.
\end{lemma}

Versions of this lemma have appeared many times in the literature; see, for instance, \cite[Lemma 2.1]{HL2014}.

%We include a proof for completeness, and because it is relatively simple.

%\begin{proof}[Proof (of Lemma~\ref{lem: resc curve})] Let $\gamma_{\circ} \colon \R \to \R^n$ denote the moment curve
%\begin{equation}\label{eq: resc curve 1}
%    \gamma_{\circ}(s) \coloneq \Big(s, \frac{s^2}{2}, \dots, \frac{s^n}{n!} \Big), \qquad s \in \R. 
%\end{equation}
%By Taylor's theorem, there exists a smooth function $\cE_P \colon [-1,1] \to \R^n$ such that
%\begin{equation}\label{eq: resc curve 2}
%    \gamma \big(s_P + \eta^{\frac{1}{n-1}}s \big) - \gamma(s_P) = \eta^{\frac{n}{n-1}} [\gamma]_P \circ \bdD_P \gamma_{\circ}(s) + \cE_P(s) 
%\end{equation}
%and
%\begin{equation}\label{eq: resc curve 3}
%    |\cE_P^{(r)}(s)|_{\infty} \leq \frac{B}{(n-r+1)!} \eta^{\frac{n+1}{n-1}} \qquad \textrm{for all $s \in [-1,1]$.}
%\end{equation}
%Note that 
%\begin{equation*}
%  \|\big([\gamma]_P\big)^{-1}\|_{\ell^{\infty}(\txM_n(\R))} \leq B^n,
%\end{equation*}
%where the above norm on the space $\txM_n(\R)$ of $n \times n$ matrices is given by taking the maximum absolute value of the entries. In light of this, \eqref{eq: resc curve 1} and the definition of $\gamma_P$ from \eqref{eq: rescaling 2} and \eqref{eq: resc curve 2} and \eqref{eq: resc curve 3}, we deduce that
%\begin{equation*}
%  |\gamma_P^{(r)}(s) - \gamma_{\circ}^{(r)}(s)|_{\infty} \leq n B^{n+1}\eta^{\frac{1}{n-1}}.
%\end{equation*}
%Since $\gamma_{\circ} \in \fkG(1)$, the desired result now follows. 
%\end{proof}

\begin{lemma}\label{lem: rescaling} With the above setup, for $P \in \bbP[I]$ and $T \in \T[I; P]$, let $\wt{T}$ denote the $\delta/\eta$-plate associated to $\gamma_P$ centred at $x_{\wt{T}} \coloneq \bdA_P(x_T)$ with direction $s_{\wt{T}}$ satisfying
\[ s_T = s_P + \eta^{\frac{1}{n-1}}s_{\wt{T}}. \]
Then $\bdA_P(T) = \wt{T}$. 
\end{lemma}

\begin{proof} Under the hypotheses of the lemma, a simple computation shows that 
\begin{equation}\label{eq: rescaling 3}
    [\gamma_P]_{\wt{T}} = \bdD_P \circ \big([\gamma]_P\big)^{-1} \circ [\gamma]_T \circ (\bdD_P)^{-1}
\end{equation}

The elements $x \in T$ are precisely those of the form 
\begin{equation*}
    x = x_T + \big([\gamma]_T\big)^{-\top} \bdD_T \,y \qquad \textrm{for some $y \in [-1,1]^n$.}
\end{equation*}
Thus, applying $\bdA_P$, we see that $x \in T$ if and only if 
\begin{equation*}
    \bdA_P(x) = \bdA_P(x_T) + \big(\bdD_P\big)^{-1} \circ \big([\gamma]_P\big)^{\top} \circ \big([\gamma]_T\big)^{-\top} \circ \bdD_T \,y \quad \textrm{for some $y \in [-1,1]^n$.}
\end{equation*}
Using \eqref{eq: rescaling 3}, we may write
\begin{align*}
   \big(\bdD_P\big)^{-1} \circ\big([\gamma]_P\big)^{\top} \circ \big([\gamma]_T\big)^{-\top} &= \Big(\big(\bdD_P\big) \circ\big([\gamma]_P\big)^{-1} \circ [\gamma]_T  \circ \big(\bdD_P\big)^{-1} \Big)^{-\top} \circ \big(\bdD_P\big)^{-1} \\
   &= \big( [\gamma_P]_{\wt{T}}  \big)^{-\top} \circ \big(\bdD_P\big)^{-1}
\end{align*}
Thus, combining the previous two displays, we deduce that $x \in T$ if and only if 
\begin{equation*}
    \bdA_P(x) = \bdA_P(x_T) + \big(\bdD_P\big)^{-1} \big( [\gamma_P]_{\wt{T}}  \big)^{-\top} \circ \bdD_{\wt{T}} \,y \quad \textrm{for some $y \in [-1,1]^n$.}
\end{equation*}
In particular, $x \in T$ if and only if $\bdA_P(x) \in \wt{T}$, as required. 
\end{proof}

%%%%%%%%%%%%%%%%%%%%%%%%%%%%%%%%%%%%%%%%%%%%%%%%%%%%%%%%%%%%%%%%%%%%%%%%%%%%%%%%%%%%%%%%%%%%%%%%

%   The proof of the Kakeya-type estimate

%%%%%%%%%%%%%%%%%%%%%%%%%%%%%%%%%%%%%%%%%%%%%%%%%%%%%%%%%%%%%%%%%%%%%%%%%%%%%%%%%%%%%%%%%%%%%%%%

\subsection{The proof of the Kakeya-type estimate}  We are now ready to prove Proposition~\ref{prop: Kakeya}, following a standard broad-narrow argument as per \cite{BG2011, HL2014}. 

\begin{proof}[Proof (of Proposition~\ref{prop: Kakeya})] Fix $B \geq 2$ and $s > s_n$ and let $\bdC_{B, s} \geq 1$ be a constant, depending only on $B$, $n$ and $s$, chosen sufficiently large so as to satisfy the requirements of the forthcoming argument. The goal is to show 
\begin{equation}\label{eq: b/n 0}
    \bigg\|\sum_{T \in \T} \chi_{T} \bigg\|_{L^s(\R^n)} \leq \bdC_{B, s} \delta^{e(n;s)} \big[ \# \T\big]^{n/s}
\end{equation}
holds whenever $\gamma \in \fkG(B)$ and $\T$ is a family of essentially distinct $\delta$-plates associated to~$\gamma$. The desired result (which involves coefficients $(\nu_T)_{T \in \T}$) then follows from restricted weak-type interpolation, using the fact we are working with an open range of estimates. 

We argue by induction on $\delta$. For $1 \geq \delta > 1/100$, provided $\bdC_{B, s}$ is chosen sufficiently large, the desired bound \eqref{eq: b/n 0} trivially holds. This serves as the base case for the induction.

Fix $0 < \delta \leq 1/100$ and assume the following. \medskip

\noindent \textbf{Induction hypothesis.} For all $2\delta \leq \tilde{\delta} \leq 1$, the inequality
\begin{equation}\label{eq: ind hyp}
    \bigg\|\sum_{\wt{T} \in \wt{\T}} \chi_{\wt{T}} \bigg\|_{L^s(\R^n)} \leq \bdC_{B, s} \tilde{\delta}^{e(n;s)} \big[ \# \wt{\T}\big]^{n/s}
\end{equation}
holds for $\tilde{\gamma} \in \fkG(B)$ and a family $\wt{\T}$ of essentially distinct $\tilde{\delta}$-plates associated to $\tilde{\gamma}$.\medskip 

Fix $\gamma \in \fkG(B)$ and a direction-separated family of $\delta$-plates $\T$ associated to $\gamma$. Let $0 < \rho_B < 1$ be the constant appearing in \eqref{eq: non-deg trans}. By pigeonholing, we may find some interval $I_{\circ} \subseteq [-1,1]$ of radius $\rho_B/2$ such that 
\begin{equation}\label{eq: b/n 0.5}
    \bigg\|\sum_{T \in \T} \chi_{T} \bigg\|_{L^s(\R^n)} \lesssim_B \bigg\|\sum_{T \in \T[I_{\circ}]} \chi_{T} \bigg\|_{L^s(\R^n)}.   
\end{equation}

Now let $0 < \eta < \min\{c_B, \rho_B/2\}$, where $0 < c_B < 1$ is the constant with the property guaranteed by Lemma~\ref{lem: resc curve}. Let $\cI(\eta)$ be a boundedly overlapping cover of $I_{\circ}$ by intervals of length $\eta^{1/(n-1)}$ and define
\begin{equation*}
    \cI_{\mathrm{broad}}^n(\eta) \coloneq \left\{ (I_1, \dots, I_n) \in \cI(\eta)^n : \min_{1 \leq j < k \leq n} \dist(I_j, I_k) > \eta^{\frac{1}{n-1}} \right\}. 
\end{equation*}
We apply a standard broad-narrow decomposition (see, for instance, \cite[Lemma 2.8]{HL2014}) to pointwise dominate
\begin{align}
\nonumber
    \sum_{T \in \T[I_{\circ}]} \chi_T(x) &\leq \sum_{I \in \cI(\eta)} \sum_{T \in \T[I]} \chi_T(x) \\ 
    & \lesssim \left(\sum_{I \in \cI(\eta)} \bigg| \sum_{T \in \T[I]} \chi_T(x)\bigg|^s\right)^{1/s} \nonumber\\
    &\quad\quad\quad\quad\quad + \sum_{(I_1, \dots, I_n) \in \cI_{\mathrm{broad}}^n(\eta) } \prod_{\ell = 1}^n \bigg| \sum_{T_{\ell} \in \T[I_{\ell}]} \chi_{T_{\ell}}(x) \bigg|^{1/n} \label{eq: b/n 1}.
\end{align}

If $(I_1, \dots, I_n) \in \cI_{\mathrm{broad}}^n(\eta)$, then it follows from \eqref{eq: non-deg trans} that $\T[I_1], \dots, \T[I_n] \subseteq \T$ are $\Omega(\eta^{n/2})$-transversal. Thus, by the multilinear estimate from Proposition~\ref{prop: n linear}, we have
\begin{equation}\label{eq: b/n 2}
    \bigg\| \prod_{\ell=1}^n \bigg|\sum_{T_{\ell} \in \T_{\ell}} \chi_{T_{\ell}}\bigg|^{1/n}\bigg\|_{L^s(\R^n)} \lesssim \eta^{- \frac{n}{2s}} \delta^{e(n; s)}[\#\T]^{n/s}.
\end{equation}
Summing over all $(I_1, \dots, I_n) \in \cI_{\mathrm{broad}}^n(\eta)$, we deduce that there exists some $E_n > 0$ such that
\begin{equation*}
    \sum_{(I_1, \dots, I_n) \in \cI_{\mathrm{broad}}^n(\eta) }\bigg\| \prod_{\ell=1}^n \bigg|\sum_{T_{\ell} \in \T_{\ell}} \chi_{T_{\ell}}\bigg|^{1/n}\bigg\|_{L^s(\R^n)} \lesssim \eta^{-E_n} \delta^{e(n; s)}[\#\T]^{n/s}.
\end{equation*}
This provides a bound for the contributions arising from the second term on the right-hand side of \eqref{eq: b/n 1}. 

We now turn to estimating the contributions arising from the first term on the right-hand side of \eqref{eq: b/n 1}. Here we use a combination of rescaling and the induction hypothesis. First note that, by Lemma~\ref{lem: Kakeya orthog} and the bounded overlap between parallel $\overline{P}$, we have
\begin{equation}\label{eq: b/n 3}
   \bigg\|\sum_{T \in \T[I]} \chi_T \bigg\|_{L^s(\R^n)} \lesssim_B  \left(\sum_{P \in \bbP[I]} \bigg\|\sum_{T \in \T[I; P]} \chi_T \bigg\|_{L^s(\overline{P})}^s\right)^{1/s}.
\end{equation}
Fix $P \in \bbP[I]$ and let $\tilde{\delta} \coloneq \delta/\eta$ so that $2 \delta \leq \tilde{\delta} \leq 1$. To estimate the summand in \eqref{eq: b/n 3} corresponding to a fixed $P$, we apply a change of variables via the affine transformation $\bdA_P$ from \eqref{eq: rescaling 1}. In particular, in view of Lemma~\ref{lem: rescaling}, we deduce that
\begin{equation*}
    \bigg\|\sum_{T \in \T[I; P]} \chi_T \bigg\|_{L^s(\overline{P})} \lesssim_B \eta^{\frac{n}{2s}} \bigg\|\sum_{\wt{T} \in \wt{\T}[I; P]} \chi_{\wt{T}} \bigg\|_{L^s(\R^n)}
\end{equation*}
where $\wt{\T}$ is a direction-separated family of $\tilde{\delta}$-plates associated to some non-degenerate curve $\tilde{\gamma}$. Furthermore, by Lemma~\ref{lem: resc curve}, we know $\tilde{\gamma} \in \fkG(2) \subseteq \fkG(B)$. Thus, by the induction hypothesis \eqref{eq: ind hyp}, we have
\begin{align}
\nonumber
    \bigg\|\sum_{T \in \T[I; P]} \chi_T \bigg\|_{L^s(\overline{P})} &\lesssim_B \bdC_{B, s} \eta^{n/(2s)} \tilde{\delta}^{e(n;s)} \big[\#\wt{\T}[I; P]\big]^{n/s} \\
    \label{eq: b/n 4}
    &= \eta^{\varepsilon(s)} \bdC_{B, s}  \delta^{e(n;s)} \big[\# \T[I; P]\big]^{n/s}
\end{align}
where
\begin{equation*}
    \varepsilon(s) \coloneq \frac{n}{2}\cdot \frac{1}{s} - e(n;s) = \frac{1}{n-1} \Big( 1 - \frac{n(n+1)}{2} \cdot \frac{1}{s} \Big) > 0 \quad \textrm{for} \quad s > \frac{n(n+1)}{2}.
\end{equation*}

Recalling \eqref{eq: b/n 3}, if we take the $\ell^s$-norm of both sides of \eqref{eq: b/n 4} with respect to both $P \in \bbP[I]$ and $I \in \cI(\eta)$, it then follows that
\begin{align}
\nonumber
    \left(\sum_{I \in \cI(\eta)} \Big\|\sum_{T \in \T[I]} \chi_T \Big\|_{L^s(\R^n)}^s \right)^{1/s} &\lesssim_B \eta^{\varepsilon(s)} \bdC_{B, s}  \delta^{e(n;s)} \bigg( \sum_{I \in \cI(\eta)} \sum_{P \in \bbP[I]} \big[\# \T[I; P]\big]^n \bigg)^{1/s} \\
     \label{eq: b/n 5}
     &\lesssim_B \eta^{\varepsilon(s)} \bdC_{B, s}  \delta^{e(n;s)} [\#\T]^{n/s},
\end{align}
where the last step follows from the nesting of $\ell^p$.

We now take the $L^s(\R^n)$-norm of both sides of \eqref{eq: b/n 1}, apply the triangle inequality and use \eqref{eq: b/n 5} and \eqref{eq: b/n 2} to bound the resulting expressions. It then follows that there exists a constant $C_B \geq 1$, depending only on $B \geq 1$, such that
\begin{equation}\label{eq: b/n 6}
    \bigg\|\sum_{T \in \T[I_{\circ}]} \chi_T \bigg\|_{L^s(\R^n)} \leq C_B \big( \eta^{\varepsilon(s)}\bdC_{B, s} + \eta^{-E_n} \big) \delta^{e(n;s)} [\#\T]^{n/s}. 
\end{equation}
In the above argument, we are free to choose $0 < \eta < \min\{c_B, \rho_B/2\}$. In particular, we may ensure $\eta$ is sufficiently small, depending only on $B$ and $s$, so as to satisfy $C_B \eta^{\varepsilon(s)} \leq 1/2$. On the other hand, we may also ensure $\bdC_{B, s} \geq 1$ is chosen sufficiently large, again depending only on $B$ and $s$, so that $C_B \eta^{-E_n} \leq \bdC_{B, s}/2$. With these choices, recalling \eqref{eq: b/n 0.5}, the bound \eqref{eq: b/n 6} implies
\begin{equation*}
    \bigg\|\sum_{T \in \T} \chi_T \bigg\|_{L^s(\R^n)} \leq \bdC_{B, s} \delta^{e(n;s)} [\#\T]^{n/s}, 
\end{equation*}
which closes the induction and completes the proof. 
\end{proof}

%%%%%%%%%%%%%%%%%%%%%%%%%%%%%%%%%%%%%%%%%%%%%%%%%%%%%%%%%%%%%%%%%%%%%%%%%%%%%%%%%%%%%%%%%%%%%%%%

%                                         Comments and further results

%%%%%%%%%%%%%%%%%%%%%%%%%%%%%%%%%%%%%%%%%%%%%%%%%%%%%%%%%%%%%%%%%%%%%%%%%%%%%%%%%%%%%%%%%%%%%%%%

\subsection{Comments and further results} Although Proposition~\ref{prop: Kakeya} is stated for essentially disjoint plates, Theorem~\ref{thm: square fn} is only stated for wave packets associated to the more restrictive class of direction-separated plates. This is simply because wave packets associated to essentially disjoint plates do not, in general, form orthonormal systems. For $n = 2$, they do form `almost' orthonormal systems, which allows one to prove a variant of Theorem~\ref{thm: square fn} for essentially disjoint plates in this case. However, the almost orthonormal property no longer holds in higher dimensions. 
 
 It is also natural to consider the geometric inequality in the situation where the essentially distinct condition is strengthened to direction-separation. Under this hypothesis, a larger range of estimates holds. 

\begin{proposition}\label{prop: true Kakeya} Let $1 \leq r \leq s \leq \infty$ satisfy $\frac{n(n-1)}{2} \frac{1}{s} > 1 - \frac{1}{r}$. Then the inequality

\begin{equation*}
    \bigg\|\sum_{T \in \T} \chi_T \bigg\|_{L^s(\R^n)} \lesssim_{\gamma, r, s} \delta^{n/2s} \big[ \# \T\big]^{1/r}
\end{equation*}
whenever $\T$ is a direction-separated family of $\delta$-plates associated to $\gamma$. 
\end{proposition}

Proposition~\ref{prop: true Kakeya} follows from a minor modification of the argument used to prove Proposition~\ref{prop: Kakeya}. Note that any direction-separated family $\T$ of $\delta$-plates satisfies $\# \T \lesssim \delta^{-1/(n-1)}$. This property yields an improved $L^\infty$ bound to that of Lemma~\ref{lem: bdd mult}, which yields an improved multilinear estimate when interpolated with the $L^n(\R^n)$ multilinear estimate in~\ref{prop: n linear}. Following through the proof of Proposition~\ref{prop: Kakeya} yields the claim. We leave the details to the interested reader.

%%%%%%%%%%%%%%%%%%%%%%%%%%%%%%%%%%%%%%%%%%%%%%%%%%%%%%%%%%%%%%%%%%%%%%%%%%%%%%%%%%%%%%%%%%%%%%%%

%                                          REFERENCES

%%%%%%%%%%%%%%%%%%%%%%%%%%%%%%%%%%%%%%%%%%%%%%%%%%%%%%%%%%%%%%%%%%%%%%%%%%%%%%%%%%%%%%%%%%%%%%%%

\bibliography{Reference}
\bibliographystyle{amsplain}

\end{document}